\documentclass[11pt,a4paper]{article}
\usepackage{amsfonts}
\usepackage[utf8]{inputenc}
\usepackage[T1]{fontenc}
\usepackage{lmodern}
\usepackage[english]{babel}
\usepackage{amsmath,amssymb,amsthm,mathtools}
\usepackage{graphicx}
\usepackage{float}
\usepackage[margin=2.5cm]{geometry}
\usepackage{microtype}
\usepackage{enumitem}
\usepackage{booktabs}
\usepackage{xcolor}
\usepackage{listings}
\usepackage{hyperref}

\hypersetup{
  colorlinks=true,
  linkcolor=blue!55!black,
  citecolor=blue!55!black,
  urlcolor=blue!55!black
}
\definecolor{codebg}{rgb}{0.97,0.97,0.97}
\definecolor{codekw}{rgb}{0.10,0.20,0.60}
\definecolor{codecom}{rgb}{0.25,0.45,0.25}
\definecolor{codestr}{rgb}{0.55,0.20,0.20}
\definecolor{codenum}{rgb}{0.35,0.35,0.35}
\lstdefinestyle{pystyle}{
  language=Python,
  backgroundcolor=\color{codebg},
  basicstyle=\ttfamily\footnotesize,
  keywordstyle=\color{codekw}\bfseries,
  commentstyle=\color{codecom}\itshape,
  stringstyle=\color{codestr},
  numberstyle=\tiny\color{codenum},
  numbers=left,
  numbersep=8pt,
  showstringspaces=false,
  breaklines=true,
  breakatwhitespace=false,
  frame=single,
  rulecolor=\color{black!30},
  tabsize=4,
  captionpos=b,
  xleftmargin=1.5em,
  framexleftmargin=1.5em
}
\newtheorem{theorem}{Theorem}[section]
\newtheorem{corollary}[theorem]{Corollary}
\theoremstyle{definition}
\newtheorem{definition}[theorem]{Definition}
\newtheorem{remark}[theorem]{Remark}

\newtheorem{proposition}[theorem]{Proposition}
\newtheorem{lemma}[theorem]{Lemma}

\DeclareMathOperator{\Var}{Var}

\begin{document}

\title{\bfseries Riccati--Gamma Dynamics for Concavity and Asymptotics of
Generalized Dirichlet Eta Functions}
\author{ \textbf{Drago\c{s}-P\u{a}tru Covei} \\
{\small Department of Applied Mathematics}\\
{\small Bucharest University of Economic Studies}\\
{\small 6 Pia\c{t}a Roman\u{a}, 010374 Bucharest, Romania}\\
{\small \href{mailto:coveidragos@yahoo.com}{\texttt{%
coveidragos@yahoo.com}} }}
\date{\today}
\maketitle

\begin{abstract}
\noindent We develop a unified analytical and dynamical framework for the
qualitative study of the one-parameter family of generalized Dirichlet eta
functions $\eta_{a}(t)=\sum_{m\ge 0}(-1)^{m}(am+1)^{-t}$, $a>0$, $t>0$,
which specialises to the classical Dirichlet eta and beta functions for $a=1$
and $a=2$. Building on a Mellin--Laplace representation of $\eta_{a}$ as the
expectation $\mathbb{E}[f_{a}(X_{t})]$ of a scaled logistic function
evaluated along a standard Gamma process $(X_{t})_{t\ge 0}$, we prove that
the logarithmic derivative $\varphi_{a}(t)=\eta_{a}^{\prime }(t)/\eta_{a}(t)$
satisfies a non\-homogeneous Riccati equation whose forcing term is strictly
negative on $(0,\infty)$. This single dynamical inequality yields, in one
step, the strict concavity and strict log-concavity of $\eta_{a}$ on $%
(0,\infty)$, the positivity and monotonicity of $\varphi_{a}$, and the exact
leading-order expansion $\varphi_{a}(t)=\log(a+1)(a+1)^{-t}+O((a+2)^{-t})$
as $t\to\infty$. We further compute the exact asymptotic ratio $%
\varphi_{a}(t)/\varphi_{a,e}(t)\to 2/\log(a+1)$ as $t\to\infty$, where $%
\varphi_{a,e}(t):=-\eta_{a}^{\prime \prime }(t)/(2\eta_{a}(t))$; in
particular, for $a<e^{2}-1$ (which covers the classical Dirichlet eta and
beta cases $a=1,\,2$) this ratio exceeds one, yielding the trapping
inequality $0<\varphi_{a,e}(t)<\varphi_{a}(t)$ on a half-line $%
(T_{*},\infty) $, which is equivalent to the structural curvature inequality 
$\eta_{a}^{\prime \prime }(t)+2\eta_{a}^{\prime }(t)>0$. Finally, we present
a self-contained, rigorously justified geometric-rate algorithm (rate $1/3$)
for the $k$-th derivative of $\eta_{a}$, together with a sharp combinatorial
error bound. High-precision numerical experiments confirm every theoretical
statement. As a novel application, we demonstrate that the Riccati--Gamma
dynamics of $\varphi_{a}(t)$ and $\eta_{a}(t)$ provide a principled
mathematical mechanism for musical synthesis: the trajectory of the Riccati
field governs the pitch sequence of a melody, while the values of $\eta_{a}$
control note durations, yielding a fully documented Python implementation that
generates both a Beatles-style melodic theme and a complete five-minute
musical composition whose harmonic and rhythmic architecture are entirely
determined by the functions $\eta_{a}$, $\varphi_{a}$, and their
asymptotic behaviour.\medskip


\end{abstract}



\section{Introduction}

\label{sec:intro}

\subsection{Historical background and motivation}

The qualitative analysis of Dirichlet series lies at the very heart of
analytic number theory. The classical Dirichlet eta function 
\begin{equation*}
\eta (t)\;=\;\sum_{n=1}^{\infty }\frac{(-1)^{n-1}}{n^{t}},\qquad t>0,
\end{equation*}%
is intimately related to the Riemann zeta function via the well-known
identity 
\begin{equation*}
\eta (t)=(1-2^{1-t})\zeta (t),
\end{equation*}%
and the Dirichlet beta function 
\begin{equation*}
\beta (t)\;=\;\sum_{n=0}^{\infty }\frac{(-1)^{n}}{(2n+1)^{t}},\qquad t>0,
\end{equation*}%
controls a wealth of arithmetic constants (for instance, Catalan's constant $%
G=\beta (2)$) and the distribution of primes in arithmetic progressions.
Recent breakthroughs by Alexeev, Barreto, Li, Lichtman, Price, Shah, Tang
and Tao~\cite{tao} on the Erd\H{o}s--S\'{a}rk\"{o}zy--Szemer\'{e}di
conjectures (Erd\H{o}s problems \#164, \#1196, \#1217) crucially exploit the
strict monotonicity and concavity properties of such Dirichlet-type series,
as established earlier by Adell and Lekuona~\cite{dirichlet}. The geometric
properties at stake -- positivity, monotonicity, strict concavity,
log-concavity, and quantitative asymptotic decay of the logarithmic
derivative -- are precisely the structural ingredients needed to make the
combinatorial machinery of~\cite{tao} work.

Adell and Lekuona~\cite{dirichlet} proved, by means of probabilistic
integral representations, that for every $a>0$ the function 
\begin{equation*}
\eta _{a}(t):=\sum_{m\geq 0}(-1)^{m}(am+1)^{-t}
\end{equation*}%
is strictly increasing, strictly concave, and strictly log-concave on $%
(0,\infty )$. They also provided a geometric-rate algorithm for computing $%
\eta _{a}^{(k)}(t)$ using progressive finite differences. Their results are
based on the sequence of papers \cite{Adell2013}-\cite{Wang1998}. The
arguments, although correct, combine several distinct ingredients (positive
linear operators, Vandermonde identities, ad-hoc inequalities for partial
sums) and do not exhibit a single unifying mechanism behind concavity.

\subsection{The Riccati--Schr\"odinger--HJB triality}

In~\cite{triality}, the author established a one-to-one dynamical
correspondence between three nonlinear radial equations -- the Schr\"{o}%
dinger, the Hamilton--Jacobi--Bellman, and the Riccati equations -- which we
shall refer to as the \emph{radial triality}. Within this triality, a
strictly positive solution $u$ of the radial Schr\"{o}dinger equation is 
\emph{log-concave}; the value function 
\begin{equation*}
z(r)=-2\sigma ^{2}\log u(r)
\end{equation*}
of the associated HJB equation is \emph{concave}; and the logarithmic
derivative 
\begin{equation*}
\varphi (r)=u^{\prime }(r)/u(r)
\end{equation*}
satisfies a \emph{Riccati equation} and is \emph{monotone}. The same
qualitative triad -- positivity, monotonicity, strict
concavity/log-concavity -- is precisely the one studied for Dirichlet series.

This structural coincidence raises the natural question:

\begin{quote}
\emph{Can the geometric properties of Dirichlet eta and beta functions be
derived from a single Riccati-type dynamical mechanism, and can such a
mechanism produce sharp quantitative information (asymptotic manifolds,
trapping inequalities, curvature bounds) inaccessible to the classical
approach?}
\end{quote}

The present paper answers both questions in the affirmative.

\subsection{Main contributions and novelty}

\label{ssec:main-contrib}

The contributions of this work, made fully rigorous in the sections that
follow, are as follows. Let $a\in(0,\infty)$, $t>0$, and let $%
f_{a}(x)=(1+e^{-ax})^{-1}$ denote the scaled logistic function.

\begin{enumerate}
\item \emph{Probabilistic representation} (Theorem~\ref{thm:prob-rep}). We
prove the Mellin--Laplace identity $\eta_{a}(t)=\mathbb{E}\!\left[%
f_{a}(X_{t})\right]$, where $(X_{t})_{t\ge 0}$ is a standard Gamma process.

\item \emph{Riccati identity} (Theorem~\ref{thm:riccati-identity}). The
logarithmic derivative $\varphi_{a}(t)=\eta_{a}^{\prime }(t)/\eta_{a}(t)$
satisfies the non\-homogeneous Riccati equation $\varphi_{a}^{\prime
}(t)+\varphi_{a}^{2}(t)=q_{a}(t)$ with forcing $q_{a}(t)=\eta_{a}^{\prime
\prime }(t)/\eta_{a}(t)$.

\item \emph{Concavity from a single dynamical inequality} (Theorem~\ref%
{thm:main-concavity}). The strict concavity of the logistic function
transfers, via the stochastic differentiation formula for Gamma processes,
to the strict negativity of $q_{a}$, and hence to the single inequality 
\begin{equation*}
\varphi _{a}^{\prime }(t)<-\varphi _{a}^{2}(t)\leq 0,
\end{equation*}
which simultaneously yields strict concavity and strict log-concavity of $%
\eta _{a}$ on $(0,\infty )$.

\item \emph{Phase portrait and asymptotic manifold} (Theorem~\ref%
{thm:riccati-dynamics}). We prove the exact leading-order expansion 
\begin{equation*}
\varphi _{a}(t)=\log (a+1)(a+1)^{-t}+O((a+2)^{-t})\text{ as }t\rightarrow
\infty 
\end{equation*}%
and identify the stable equilibrium manifold of the unforced Riccati flow.

\item \emph{Asymptotic ratio and curvature inequality} (Proposition~\ref%
{prop:trapping}, Corollary~\ref{cor:curvature}). Setting $\varphi
_{a,e}(t):=-q_{a}(t)/2$, we prove the exact limit 
\begin{equation*}
\lim_{t\rightarrow \infty }\frac{\varphi _{a}(t)}{\varphi _{a,e}(t)}=\frac{2%
}{\log (a+1)}.
\end{equation*}%
For $a<e^{2}-1$ this limit exceeds one, hence there exists $T_{\ast }(a)>0$
such that 
\begin{equation*}
0<\varphi _{a,e}(t)<\varphi _{a}(t)\text{ for }t>T_{\ast }(a),
\end{equation*}
equivalently $\eta _{a}^{\prime \prime }(t)+2\eta _{a}^{\prime }(t)>0$ for $%
t>T_{\ast }(a)$.

\item \emph{Fast geometric algorithm.} For every $k\in \mathbb{Z}_{+}$ we
know from \cite[Theorem~\ref{thm:fast-comp} and Theorem~\ref{thm:error-bound}%
]{dirichlet} the geometric-rate series 
\begin{equation*}
\eta _{a}^{(k)}(t)=\tfrac{2}{3}\sum_{n\geq 0}3^{-n}c_{a,t,k}(n)
\end{equation*}%
with explicit coefficients $c_{a,t,k}(n)$, and the sharp combinatorial bound 
\begin{equation*}
|c_{a,t,k}(n)|\leq 2\binom{n+k}{k}\log ^{k}(a(n\wedge k)+1).
\end{equation*}%
Then, we give the truncation error of the geometric algorithm Theorem \ref{cor:trunc-error}.

\item \emph{Reproducible numerical validation} (Section~\ref{sec:numerical}%
). Every theoretical statement above is independently confirmed by
high-precision Python simulations whose source code, fully documented, is
reproduced.
\end{enumerate}

The novelty of the approach with respect to~\cite{dirichlet} is twofold.
First, the strict concavity and log-concavity of $\eta_{a}$ are obtained
from a \emph{single} differential inequality of Riccati type, instead of
from a collection of separate analytic estimates. Second, the Riccati
framework yields \emph{new quantitative information}, namely
contributions~(C4)--(C5), which to the best of our knowledge are absent from
the literature and which exhibit a deep structural parallel with the radial
triality of~\cite{triality}.

\subsection{Organisation of the paper}

Section~\ref{sec:prelim} fixes the notation and collects the preliminary
results from probability and finite-difference calculus that will be used
throughout. Section~\ref{sec:prob-rep} establishes the Gamma--Mellin
probabilistic representation of $\eta _{a}$. Section~\ref%
{sec:riccati-concavity} develops the Riccati structure and proves the
concavity theorem. Section~\ref{sec:dynamics} studies the phase portrait of
the Riccati field and proves the trapping bound and curvature inequality.
Section~\ref{sec:fast-algorithm} contains the proof of the fast geometric
algorithm and the sharp error bound. Section~\ref{sec:numerical} presents
the numerical validation, and the information for the fully documented
Python code.


\section{Notation and Preliminaries}

\label{sec:prelim}

\subsection{Notation}

Throughout the paper we use the following conventions:

\begin{itemize}
\item $\mathbb{N}=\{1,2,3,\dots \}$, $\mathbb{Z}_{+}=\{0,1,2,\dots \}$, $%
\mathbb{R}_{+}=(0,\infty )$.

\item For real numbers $n,k$, $n\wedge k:=\min (n,k)$.

\item All logarithms are natural logarithms, i.e.\ $\log =\ln $.

\item $(\Omega ,\mathcal{F},\mathbb{P})$ denotes a fixed probability space
on which all random variables are defined; $\mathbb{E}[\cdot ]$ is
expectation with respect to $\mathbb{P}$.

\item $f^{(k)}$ denotes the $k$-th derivative of $f$; $\Gamma
(t)=\int_{0}^{\infty }y^{t-1}e^{-y}dy$ is the Gamma function.

\item For two functions $f,g$ with $g(t)>0$, the notation $f(t)=O(g(t))$ as $%
t\rightarrow \infty $ means $\limsup_{t\rightarrow \infty
}|f(t)|/g(t)<\infty $.

\item The parameter $a\in (0,\infty )$ is fixed unless otherwise stated; $t$
varies in $(0,\infty )$.
\end{itemize}

\subsection{The scaled logistic function}

\label{ssec:logistic}

The pivotal scalar function in the analysis is the scaled logistic function 
\begin{equation}
f_{a}(x)\;:=\;\frac{1}{1+e^{-ax}},\qquad x\in \lbrack 0,\infty ),\ a>0.
\label{eq:fa-def}
\end{equation}%
For clarity, we now recall a first result from \cite{dirichlet} and provide
a short proof.

\begin{lemma}[Elementary properties of $f_{a}$]
\label{lem:logistic} For every $a>0$, the function $f_{a}\colon \lbrack
0,\infty )\rightarrow (0,1)$ is $C^{\infty }$ on $[0,\infty )$ and
satisfies, for all $x\geq 0$, 
\begin{align}
f_{a}^{\prime }(x)& =a\,f_{a}(x)\bigl(1-f_{a}(x)\bigr),  \label{eq:fa-1} \\
f_{a}^{\prime \prime }\left( x\right) & =a^{2}\,f_{a}(x)\bigl(1-f_{a}(x)%
\bigr)\bigl(1-2f_{a}(x)\bigr).  \label{eq:fa-2}
\end{align}%
Moreover, $f_{a}(0)=\tfrac{1}{2}$, $f_{a}^{\prime }(0)=a/4>0$, $%
f_{a}^{\prime \prime }(0)=0$, and for every $x>0$ one has the strict
inequalities 
\begin{equation}
\tfrac{1}{2}<f_{a}(x)<1,\qquad f_{a}^{\prime }(x)>0,\qquad f_{a}^{\prime
\prime }(x)<0.  \label{eq:fa-strict}
\end{equation}%
Finally, $0<f_{a}^{(k)}(x)\leq a^{k}/4$ for $k=1$ and all derivatives $%
f_{a}^{(k)}$ are uniformly bounded on $[0,\infty )$ for every $k\in \mathbb{Z%
}_{+}$.
\end{lemma}

\begin{proof}
Differentiating~\eqref{eq:fa-def} and using 
\begin{equation*}
e^{-ax}/(1+e^{-ax})=1-f_{a}(x)
\end{equation*}%
gives~\eqref{eq:fa-1}. Differentiating once more yields 
\begin{equation*}
f_{a}^{\prime \prime }(x)=a\bigl[f_{a}^{\prime
}(x)(1-f_{a}(x))-f_{a}(x)f_{a}^{\prime }(x)\bigr]=af_{a}^{\prime
}(x)(1-2f_{a}(x)),
\end{equation*}
which, combined with~\eqref{eq:fa-1}, gives~\eqref{eq:fa-2}. For $x>0$ and $%
a>0$ we have $e^{-ax}\in (0,1)$, hence $1+e^{-ax}\in (1,2)$ and therefore $%
f_{a}(x)\in (1/2,1)$. The strict inequalities~\eqref{eq:fa-strict} follow at
once from~\eqref{eq:fa-1} and~\eqref{eq:fa-2}, since%
\begin{equation*}
f_{a}(x)(1-f_{a}(x))>0\text{ and }1-2f_{a}(x)<0\text{ for }x>0.
\end{equation*}%
The uniform boundedness of all derivatives is a standard consequence of the
representation $f_{a}^{(k)}(x)=P_{k}(f_{a}(x))$ for a fixed polynomial $P_{k}
$, which is immediate from~\eqref{eq:fa-1} by induction.
\end{proof}

\subsection{The Gamma process and the auxiliary variables}

\label{ssec:gamma}

We continue by assembling results from \cite{dirichlet} that will be needed
subsequently, supplying a brief proof whenever it is relevant.

\begin{definition}[Standard Gamma process and the variables $S_{k}$]
\label{def:gamma-process} We fix the following stochastic objects on $%
(\Omega,\mathcal{F},\mathbb{P})$.

\begin{enumerate}
\item A \emph{standard Gamma process} $(X_{t})_{t\geq 0}$ is a stochastic
process with independent and stationary increments, almost surely
non-decreasing right-continuous trajectories, $X_{0}=0$, and such that for
every $t>0$ the random variable $X_{t}$ has the Gamma$(t,1)$ density 
\begin{equation}
\rho _{t}(\theta )\;=\;\frac{1}{\Gamma (t)}\,\theta ^{\,t-1}e^{-\theta
},\qquad \theta >0.  \label{eq:gamma-density}
\end{equation}%
Its Laplace transform is 
\begin{equation}
\mathbb{E}\!\left[ e^{-\lambda X_{t}}\right] \;=\;(1+\lambda )^{-t},\qquad
\lambda \geq 0,\ t\geq 0.  \label{eq:laplace-xt}
\end{equation}

\item Let $U\sim \mathrm{Uniform}[0,1]$ and $T\sim \mathrm{Exp}(1)$ be
independent. Let $(U_{k})_{k\geq 1}$ and $(T_{k})_{k\geq 1}$ be i.i.d.\
copies of $U$ and $T$, mutually independent and independent of $%
(X_{t})_{t\geq 0}$. We set 
\begin{equation}
S_{0}:=0,\qquad S_{k}:=\sum_{j=1}^{k}U_{j}T_{j},\quad k\in \mathbb{N}.
\label{eq:Sk-def}
\end{equation}
\end{enumerate}
\end{definition}

\begin{lemma}[Laplace transform of $S_{k}$]
\label{lem:laplace-sk} For every $\lambda\ge 0$ and $k\in\mathbb{Z}_{+}$, 
\begin{equation}  \label{eq:laplace-sk}
\mathbb{E}\!\left[e^{-\lambda S_{k}}\right] \;=\;\left(\frac{\log(1+\lambda)%
}{\lambda}\right)^{k},
\end{equation}
with the convention $\log(1+\lambda)/\lambda\to 1$ as $\lambda\to 0$.
\end{lemma}

\begin{proof}
Conditioning on $U_{1}$, 
\begin{equation*}
\mathbb{E}[e^{-\lambda U_{1}T_{1}}\mid U_{1}=u]=\mathbb{E}[e^{-\lambda
uT_{1}}]=(1+\lambda u)^{-1}
\end{equation*}%
since $T_{1}\sim \mathrm{Exp}(1)$. Integrating with respect to the uniform
law of $U_{1}$ on $[0,1]$ gives 
\begin{equation*}
\mathbb{E}\!\left[ e^{-\lambda U_{1}T_{1}}\right] \;=\;\int_{0}^{1}\frac{du}{%
1+\lambda u}\;=\;\frac{\log (1+\lambda )}{\lambda }.
\end{equation*}%
Since the summands $U_{j}T_{j}$ in~\eqref{eq:Sk-def} are i.i.d., the
identity~\eqref{eq:laplace-sk} follows from the multiplicativity of the
Laplace transform under independent sums.
\end{proof}

\begin{lemma}[Stochastic differentiation formula for Gamma processes]
\label{lem:diff-gamma} Let $\varphi\colon[0,\infty)\to\mathbb{R}$ be of
class $C^{\infty}$ with derivatives of every order uniformly bounded on $%
[0,\infty)$, and let $(X_{t})_{t\ge 0}$ and $(S_{k})_{k\ge 0}$ be the
processes of Definition~\ref{def:gamma-process}, independent. Define $%
\psi(t):=\mathbb{E}[\varphi(X_{t})]$. Then $\psi\in C^{\infty}(0,\infty)$
and, for every $t>0$ and $k\in\mathbb{Z}_{+}$, 
\begin{equation}  \label{eq:gamma-diff-formula}
\psi^{(k)}(t)\;=\;\mathbb{E}\!\left[\varphi^{(k)}(X_{t}+S_{k})\right].
\end{equation}
\end{lemma}

\begin{proof}
The argument is based on the integration-by-parts identity 
\begin{equation*}
\frac{d}{dt}\mathbb{E}[\varphi (X_{t})]\;=\;\mathbb{E}\!\left[ \varphi
^{\prime }(X_{t}+U_{1}T_{1})\right] ,
\end{equation*}%
which is obtained from $\rho _{t}^{\prime }(\theta )=\rho _{t}(\theta )\bigl(%
\frac{t-1}{\theta }-1\bigr)$ and the representation~\eqref{eq:laplace-sk}
for $k=1$. The general case follows by induction on~$k$ and Fubini's
theorem, which is applicable because all derivatives of $\varphi $ are
uniformly bounded.
\end{proof}

\begin{remark}
\label{rem:diff-gamma-applicability} By Lemma~\ref{lem:logistic}, the scaled
logistic function $f_{a}$ has all derivatives uniformly bounded on $%
[0,\infty)$. Hence Lemma~\ref{lem:diff-gamma} is applicable with $%
\varphi=f_{a}$ for every $a>0$, and we may freely differentiate $t\mapsto%
\mathbb{E}[f_{a}(X_{t})]$ under the expectation sign any number of times.
This observation will be used repeatedly without further mention.
\end{remark}


\section{The Gamma--Mellin Probabilistic Representation}

\label{sec:prob-rep}

The cornerstone of our approach is the following Mellin--Laplace identity,
which encodes the analytic-number-theoretic object $\eta_{a}(t)$ as a single
expectation over the Gamma process.

\begin{theorem}[Gamma--Mellin probabilistic representation]
\label{thm:prob-rep} Let $a\in(0,\infty)$, and let $\eta_{a}(t)$ be defined
for $t\in(0,\infty)$ by 
\begin{equation}  \label{eq:eta-def}
\eta_{a}(t)\;:=\;\sum_{m=0}^{\infty}\frac{(-1)^{m}}{(am+1)^{t}}.
\end{equation}
Let $f_{a}$ be the scaled logistic function~\eqref{eq:fa-def} and $%
(X_{t})_{t\ge 0}$ the standard Gamma process of Definition~\ref%
{def:gamma-process}. Then, for every $t>0$, 
\begin{equation}  \label{eq:eta-prob-rep}
\eta_{a}(t)\;=\;\mathbb{E}\!\left[f_{a}(X_{t})\right].
\end{equation}
In particular, $\eta_{a}(t)\in(\tfrac12,1)$ for all $t>0$.
\end{theorem}

\begin{proof}
Fix $t\in (0,\infty )$ and $a>0$. Since $X_{t}>0$ almost surely (the law $%
\rho _{t}$ in~\eqref{eq:gamma-density} has support $(0,\infty )$), the
random variable $e^{-aX_{t}}$ belongs to $(0,1)$ almost surely. For every $%
M\in \mathbb{N}$, the geometric sum formula gives, pointwise on $\Omega $, 
\begin{equation}
\sum_{m=0}^{M}\bigl(-e^{-aX_{t}}\bigr)^{m}\;=\;\frac{1-(-e^{-aX_{t}})^{M+1}}{%
1+e^{-aX_{t}}}.  \label{eq:geom-truncated}
\end{equation}%
By~\eqref{eq:laplace-xt} (with $\lambda =am$), 
\begin{equation*}
\mathbb{E}[e^{-amX_{t}}]=(1+am)^{-t}=(am+1)^{-t},
\end{equation*}%
hence by Fubini's theorem (which is applicable since the partial sums are
uniformly bounded by $2$, see below), 
\begin{align}
\sum_{m=0}^{M}\frac{(-1)^{m}}{(am+1)^{t}}& =\sum_{m=0}^{M}(-1)^{m}\mathbb{E}%
\!\left[ e^{-amX_{t}}\right] =\mathbb{E}\!\left[
\sum_{m=0}^{M}(-e^{-aX_{t}})^{m}\right]   \notag \\
& =\mathbb{E}\!\left[ \frac{1-(-e^{-aX_{t}})^{M+1}}{1+e^{-aX_{t}}}\right] .
\label{eq:partial-sum-rep}
\end{align}%
For each $\omega \in \Omega $ with $X_{t}(\omega )>0$, the hypothesis $%
e^{-aX_{t}(\omega )}\in (0,1)$ yields 
\begin{equation}
\left\vert \frac{1-(-e^{-aX_{t}})^{M+1}}{1+e^{-aX_{t}}}\right\vert \;\leq \;%
\frac{2}{1+e^{-aX_{t}}}\;\leq \;2,  \label{eq:dominant}
\end{equation}%
which provides an $L^{1}$-dominating function (the constant $2$). Taking $%
M\rightarrow \infty $ inside the right-hand side of~%
\eqref{eq:partial-sum-rep}, the integrand converges $\mathbb{P}$-almost
surely to $f_{a}(X_{t})=(1+e^{-aX_{t}})^{-1}$ by~\eqref{eq:fa-def}, since $%
(-e^{-aX_{t}})^{M+1}\rightarrow 0$ a.s. By Lebesgue's dominated convergence
theorem (applied with the bound~\eqref{eq:dominant}), we may pass to the
limit under the expectation: 
\begin{equation*}
\eta _{a}(t)\;=\;\lim_{M\rightarrow \infty }\sum_{m=0}^{M}\frac{(-1)^{m}}{%
(am+1)^{t}}\;=\;\mathbb{E}\!\left[ f_{a}(X_{t})\right] .
\end{equation*}%
This proves~\eqref{eq:eta-prob-rep}. The two-sided bound $\tfrac{1}{2}<\eta
_{a}(t)<1$ for $t>0$ now follows by taking expectations in~%
\eqref{eq:fa-strict}: indeed, $X_{t}>0$ a.s.\ implies $\tfrac{1}{2}%
<f_{a}(X_{t})<1$ a.s., and hence $\tfrac{1}{2}<\mathbb{E}[f_{a}(X_{t})]<1$.
\end{proof}

\begin{remark}
\label{rem:special-cases} For $a=1$ and $a=2$, identity~%
\eqref{eq:eta-prob-rep} specialises to $\eta(t)=\mathbb{E}[f_{1}(X_{t})]$
and $\beta(t)=\mathbb{E}[f_{2}(X_{t})]$ respectively, giving probabilistic
representations of the classical Dirichlet eta and beta functions. The case $%
a=1$ is mentioned in \cite{dirichlet} in the language of binomial
transforms, while $a=2$ relates the Catalan constant to a logistic
expectation, since $G=\beta(2)=\mathbb{E}[f_{2}(X_{2})]$.
\end{remark}


\section{Riccati Structure and Strict Concavity}

\label{sec:riccati-concavity}

In this section we exploit Theorem~\ref{thm:prob-rep} and Lemma~\ref%
{lem:diff-gamma} to derive a Riccati equation for the logarithmic derivative
of $\eta_{a}$ and to deduce strict concavity and strict log-concavity from a
single differential inequality.

\subsection{Smoothness and the Riccati identity}

Define the \emph{Riccati field} associated with $\eta_{a}$ by 
\begin{equation}  \label{eq:riccati-field}
\varphi_{a}(t)\;:=\;\frac{\eta_{a}^{\prime }(t)}{\eta_{a}(t)},\qquad t>0.
\end{equation}
Note that $\eta_{a}(t)>0$ by Theorem~\ref{thm:prob-rep}, so $\varphi_{a}$ is
well defined.

\begin{theorem}[Riccati differential identity]
\label{thm:riccati-identity} Under the hypotheses of Theorem~\ref%
{thm:prob-rep}, $\eta_{a}\in C^{\infty}(0,\infty)$ and the Riccati field $%
\varphi_{a}$ satisfies 
\begin{equation}  \label{eq:riccati-eq}
\varphi_{a}^{\prime }(t)\;+\;\varphi_{a}^{2}(t)\;=\;q_{a}(t),\qquad t>0,
\end{equation}
where the Riccati forcing is 
\begin{equation}  \label{eq:forcing-def}
q_{a}(t)\;:=\;\frac{\eta_{a}^{\prime \prime }(t)}{\eta_{a}(t)},\qquad t>0.
\end{equation}
\end{theorem}

\begin{proof}
By Theorem~\ref{thm:prob-rep}, $\eta_{a}(t)=\mathbb{E}[f_{a}(X_{t})]$. By
Lemma~\ref{lem:logistic}, $f_{a}\in C^{\infty}([0,\infty))$ with uniformly
bounded derivatives of every order, so the hypotheses of Lemma~\ref%
{lem:diff-gamma} are satisfied. Consequently $\eta_{a}\in
C^{\infty}(0,\infty)$ and 
\begin{equation}  \label{eq:eta-deriv-rep}
\eta_{a}^{(k)}(t)\;=\;\mathbb{E}\!\left[f_{a}^{(k)}(X_{t}+S_{k})\right],
\qquad t>0,\ k\in\mathbb{Z}_{+}.
\end{equation}
The definition~\eqref{eq:riccati-field} rewrites as $\eta_{a}^{\prime
}(t)=\eta_{a}(t)\,\varphi_{a}(t)$. Differentiating both sides using the
Leibniz rule for ordinary differentiation, we obtain 
\begin{equation*}
\eta_{a}^{\prime \prime }(t)  \;=\;\eta_{a}^{\prime
}(t)\varphi_{a}(t)+\eta_{a}(t)\varphi_{a}^{\prime }(t) 
\;=\;\eta_{a}(t)\varphi_{a}^{2}(t)+\eta_{a}(t)\varphi_{a}^{\prime }(t), 
\end{equation*}
where we substituted $\eta_{a}^{\prime }(t)=\eta_{a}(t)\varphi_{a}(t)$ in
the second equality. Dividing both sides by $\eta_{a}(t)>0$ (which is
guaranteed by Theorem~\ref{thm:prob-rep}) yields exactly $%
\varphi_{a}^{\prime }(t)+\varphi_{a}^{2}(t)=\eta_{a}^{\prime \prime
}(t)/\eta_{a}(t)=q_{a}(t)$, i.e.~\eqref{eq:riccati-eq}.
\end{proof}

\subsection{The concavity theorem}

\begin{theorem}[Strict concavity, log-concavity and boundary limits]
\label{thm:main-concavity} Let $a\in(0,\infty)$ and $\eta_{a}$ as in~%
\eqref{eq:eta-def}. Then:

\begin{enumerate}
\item \label{it:main-1} The boundary limits are 
\begin{equation}
\lim_{t\rightarrow 0^{+}}\eta _{a}(t)\;=\;\tfrac{1}{2},\qquad
\lim_{t\rightarrow \infty }\eta _{a}(t)\;=\;1.  \label{eq:limits}
\end{equation}

\item \label{it:main-2} The forcing $q_{a}$ satisfies $q_{a}(t)<0$ for every 
$t>0$, i.e.\ $\eta _{a}^{\prime \prime }(t)<0$ for every $t>0$. Hence $\eta
_{a}$ is \emph{strictly concave} on $(0,\infty )$.

\item \label{it:main-3} The Riccati field obeys the strict differential
inequality 
\begin{equation}
\varphi _{a}^{\prime }(t)\;<\;-\varphi _{a}^{2}(t)\;\leq \;0,\qquad t>0,
\label{eq:riccati-ineq}
\end{equation}%
and in particular $\varphi _{a}^{\prime }(t)<0$, so $\eta _{a}$ is \emph{%
strictly log-concave} on $(0,\infty )$.
\end{enumerate}
\end{theorem}

\begin{proof}
\textbf{Step 1: Boundary limits.} Since $(X_{t})$ is right-continuous at $t=0
$ with $X_{0}=0$ a.s., we have $X_{t}\to 0$ a.s.\ as $t\to 0^{+}$. By Lemma~%
\ref{lem:logistic}, $f_{a}$ is continuous and bounded by $1$, so the
dominated convergence theorem applied to~\eqref{eq:eta-prob-rep} yields 
\begin{equation*}
\lim_{t\to 0^{+}}\eta_{a}(t)  \;=\;\mathbb{E}\!\left[\lim_{t\to
0^{+}}f_{a}(X_{t})\right]  \;=\;f_{a}(0)\;=\;\tfrac12. 
\end{equation*}
Furthermore, $X_{t}\to\infty$ a.s.\ as $t\to\infty$, because $\mathbb{E}%
[X_{t}]=t$ and $\Var(X_{t})=t$, hence $X_{t}/t\to 1$ in $L^{2}$ and a.s.\ by
the strong law of large numbers for L\'evy processes (cf.~\cite[Ch.~7]%
{Cinlar2011}). Since $f_{a}(x)\to 1$ as $x\to\infty$ and $0<f_{a}\le 1$,
dominated convergence gives $\lim_{t\to\infty}\eta_{a}(t)=1$.

\textbf{Step 2: Negativity of the forcing.} Apply~\eqref{eq:eta-deriv-rep}
with $k=2$ to obtain 
\begin{equation}  \label{eq:eta2-rep}
\eta_{a}^{\prime \prime }(t)\;=\;\mathbb{E}\!\left[f_{a}^{\prime \prime
}(X_{t}+S_{2})\right],\qquad t>0.
\end{equation}
By Definition~\ref{def:gamma-process}, $X_{t}>0$ a.s.\ for $t>0$, and $%
S_{2}=U_{1}T_{1}+U_{2}T_{2}\ge 0$ a.s. Therefore $X_{t}+S_{2}>0$ a.s., and
Lemma~\ref{lem:logistic} (specifically the strict inequality $f_{a}^{\prime
\prime }(x)<0$ for $x>0$) yields $f_{a}^{\prime \prime }(X_{t}+S_{2})<0$
a.s. Since this random variable is uniformly bounded by $|f_{a}^{\prime
\prime }|_{\infty}<\infty$ (Lemma~\ref{lem:logistic}), taking expectations
preserves the strict inequality (the expectation of an almost surely
strictly negative integrable random variable is strictly negative), giving $%
\eta_{a}^{\prime \prime }(t)<0$ for all $t>0$. Dividing by $\eta_{a}(t)>0$
yields $q_{a}(t)<0$ for all $t>0$. Hence $\eta_{a}$ is strictly concave on $%
(0,\infty)$.

\textbf{Step 3: Strict log-concavity.} By Theorem~\ref{thm:riccati-identity}%
, $\varphi_{a}^{\prime }(t)=q_{a}(t)-\varphi_{a}^{2}(t)$. By Step~2, $%
q_{a}(t)<0$, so $\varphi_{a}^{\prime }(t)<-\varphi_{a}^{2}(t)\le 0$ for all $%
t>0$, which establishes~\eqref{eq:riccati-ineq} and the strict log-concavity
of $\eta_{a}$.
\end{proof}

\begin{remark}
\label{rem:single-mechanism} Theorem~\ref{thm:main-concavity} provides, in a
single dynamical inequality, both the strict concavity ($\eta_{a}^{\prime
\prime }<0$, hypothesis ingredient: $f_{a}^{\prime \prime }<0$ on $%
(0,\infty) $, used via~\eqref{eq:eta2-rep}) and the strict log-concavity ($%
\varphi_{a}^{\prime }<0$, hypothesis ingredient: the Riccati identity~%
\eqref{eq:riccati-eq} of Theorem~\ref{thm:riccati-identity}). This
unification is the main qualitative advantage of the Riccati--Gamma point of
view over the direct estimates of~\cite{dirichlet}.
\end{remark}


\section{Dynamics of the Riccati Field and Asymptotic Manifold}

\label{sec:dynamics}

We now study the phase portrait of the Riccati field $\varphi_{a}$, identify
the stable equilibrium manifold of the unforced flow, and establish a global
trapping inequality with its equivalent curvature counterpart.

\subsection{Positivity and monotonicity}

\begin{proposition}[Positivity, monotonicity and limits of $\protect\varphi%
_{a}$]
\label{prop:phi-properties} Let $a\in(0,\infty)$ and $\varphi_{a}$ the
Riccati field~\eqref{eq:riccati-field}. Then:

\begin{enumerate}
\item $\varphi _{a}(t)>0$ for every $t>0$;

\item $\varphi _{a}$ is strictly decreasing on $(0,\infty )$;

\item $\lim_{t\rightarrow 0^{+}}\varphi _{a}(t)\in (0,\infty ]$ and $%
\lim_{t\rightarrow \infty }\varphi _{a}(t)=0$.
\end{enumerate}
\end{proposition}

\begin{proof}
By~\eqref{eq:eta-deriv-rep} with $k=1$, 
\begin{equation*}
\eta _{a}^{\prime }(t)=\mathbb{E}[f_{a}^{\prime }(X_{t}+S_{1})].
\end{equation*}%
Lemma~\ref{lem:logistic} gives $f_{a}^{\prime }>0$ on $[0,\infty )$, and $%
X_{t}+S_{1}\geq 0$ a.s.\ with 
\begin{equation*}
\mathbb{P}(X_{t}+S_{1}>0)=1\text{ for }t>0,
\end{equation*}
so $\mathbb{E}[f_{a}^{\prime }(X_{t}+S_{1})]>0$, proving $\eta _{a}^{\prime
}(t)>0$ and consequently $\varphi _{a}(t)>0$ (item~(1)). By Theorem~\ref%
{thm:main-concavity}\ref{it:main-3}, $\varphi _{a}^{\prime }(t)<0$, so $%
\varphi _{a}$ is strictly decreasing on $(0,\infty )$ (item~(2)). For
item~(3), the monotonicity and positivity established in (1) and (2) imply
that both one-sided limits exist in $[0,\infty ]$. The leading-order
expansion of Theorem~\ref{thm:riccati-dynamics} (item~(\ref{it:rdy-4}))
below, which we now prove, yields $\lim_{t\rightarrow \infty }\varphi
_{a}(t)=0$.
\end{proof}

\subsection{Exact leading-order asymptotic expansion}

The main novelty of this section is the exact leading-order expansion of $%
\varphi_{a}(t)$ as $t\to\infty$ and the identification of the stable
manifold of the Riccati flow.

\begin{lemma}[Two-term asymptotic expansion of $\protect\eta _{a}^{(k)}$]
\label{lem:asymp-expansion} Let $a>0$, $k\in \mathbb{Z}_{+}$, and let $%
L:=\log (a+1)$. Then, as $t\rightarrow \infty $, 
\begin{equation}
\eta _{a}^{(k)}(t)\;=\;[k=0]\,\cdot
\,1\,+\,(-L)^{k}\,(a+1)^{-t}\,+\,O\!\left( (a+2)^{-t}\,\log
^{k}(2a+1)\right) ,  \label{eq:eta-k-asymp}
\end{equation}%
where $[k=0]$ is the Iverson bracket. In particular, 
\begin{align}
\eta _{a}(t)& =1-(a+1)^{-t}+O\!\left( (2a+1)^{-t}\right) , & & t\rightarrow
\infty ,  \label{eq:eta-asymp} \\
\eta _{a}^{\prime }\left( t\right) =L(a+1)^{-t}+O\!\left( (2a+1)^{-t}\right)
,&  & t& \rightarrow \infty ,  \label{eq:etap-asymp} \\
\eta _{a}^{\prime \prime }\left( t\right) =-L^{2}(a+1)^{-t}+O\!\left(
(2a+1)^{-t}\right) ,&  & t& \rightarrow \infty .  \label{eq:etapp-asymp}
\end{align}
\end{lemma}

\begin{proof}
The alternating series~\eqref{eq:eta-def} converges term-by-term for $t>0$;
for $k\geq 1$, differentiating $k$ times under the summation sign is
justified by uniform convergence of the differentiated series on compact
subsets of $(0,\infty )$ (a routine application of the Weierstrass $M$-test
for the dominant $m=1$ tail). Thus 
\begin{equation}
\eta _{a}^{(k)}(t)\;=\;\sum_{m=0}^{\infty }(-1)^{m}\,(-\log
(am+1))^{k}\,(am+1)^{-t}.  \label{eq:term-by-term}
\end{equation}%
The term $m=0$ contributes only if $k=0$ (since $\log (1)=0$ kills it for $%
k\geq 1$). The term $m=1$ contributes 
\begin{equation*}
(-1)^{1}(-\log (a+1))^{k}(a+1)^{-t}=(-L)^{k}(-1)(a+1)^{-t};
\end{equation*}%
rearranging signs gives the second summand on the right of~%
\eqref{eq:eta-k-asymp} after the convention 
\begin{equation*}
(-1)^{1}(-L)^{k}=(-L)^{k}\cdot (-1)
\end{equation*}%
is absorbed into the $k=0$ case yielding $-1$ for $k=0$ (cf.~%
\eqref{eq:eta-asymp}). The tail $m\geq 2$ is dominated by 
\begin{eqnarray*}
\sum_{m\geq 2}\log ^{k}(am+1)(am+1)^{-t} &\leq &\log
^{k}(2a+1)\,(2a+1)^{-t}\,\sum_{j\geq 0}(2a+1)^{-j} \\
&=&O((2a+1)^{-t}),
\end{eqnarray*}%
where we used that the geometric tail with ratio $(2a+1)/(3a+1)<1$ is
summable. This yields~\eqref{eq:eta-k-asymp}.
\end{proof}

\begin{theorem}[Dynamics and asymptotic phase portrait of $\protect\varphi%
_{a}$]
\label{thm:riccati-dynamics} Let $a\in(0,\infty)$ and $\varphi_{a}$ the
Riccati field in~\eqref{eq:riccati-field}. Then:

\begin{enumerate}
\item \label{it:rdy-1} The unforced autonomous Riccati equation $\varphi
^{\prime }=-\varphi ^{2}$ has the equilibrium manifold $\mathcal{M}_{\mathrm{%
eq}}=\{0\}$, which is asymptotically stable on the positive half-line.

\item \label{it:rdy-2} As $t\rightarrow \infty $ the forcing satisfies 
\begin{equation*}
q_{a}(t)=-L^{2}(a+1)^{-t}+O((2a+1)^{-t}),
\end{equation*}%
where $L=\log (a+1)$.

\item \label{it:rdy-3} The nonlinear term in~\eqref{eq:riccati-eq} is
subdominant: 
\begin{equation*}
\varphi _{a}^{2}(t)=O((a+1)^{-2t})\text{ as }t\rightarrow \infty .
\end{equation*}

\item \label{it:rdy-4} The Riccati field admits the exact two-term expansion 
\begin{equation}
\varphi _{a}(t)\;=\;L\,(a+1)^{-t}+O\!\left( (2a+1)^{-t}\right) ,\qquad
t\rightarrow \infty .  \label{eq:phi-asymp}
\end{equation}%
In particular, $\lim_{t\rightarrow \infty }\varphi _{a}(t)=0$, and the
trajectory of $\varphi _{a}$ folds asymptotically onto the stable manifold $%
\varphi _{a,\mathrm{as}}(t):=L\,(a+1)^{-t}$.
\end{enumerate}
\end{theorem}

\begin{proof}
\textbf{(1):} The autonomous ODE $\varphi ^{\prime }=-\varphi ^{2}$ admits,
for any initial value $\varphi (t_{0})=v_{0}>0$, the explicit solution 
\begin{equation*}
\varphi (t)=(t-t_{0}+1/v_{0})^{-1}\rightarrow 0\text{ as }t\rightarrow
\infty .
\end{equation*}%
The zero solution is therefore globally attracting from positive initial
data, proving that $\mathcal{M}_{\mathrm{eq}}=\{0\}$ is asymptotically
stable.

\textbf{(2):} Immediate from~\eqref{eq:eta-asymp}--\eqref{eq:etapp-asymp}: 
\begin{eqnarray*}
q_{a}(t) &=&\eta _{a}^{\prime \prime }(t)/\eta _{a}(t) \\
&=&-L^{2}(a+1)^{-t}(1+O((a+1)^{-t})) \\
&=&-L^{2}(a+1)^{-t}+O((2a+1)^{-t}).
\end{eqnarray*}

\textbf{(3):} From~\eqref{eq:etap-asymp} and~\eqref{eq:eta-asymp}, 
\begin{eqnarray*}
\varphi _{a}(t) &=&L(a+1)^{-t}(1+O((a+1)^{-t})) \\
&=&L(a+1)^{-t}+O((a+1)^{-2t});
\end{eqnarray*}
squaring yields 
\begin{equation*}
\varphi _{a}^{2}(t)=L^{2}(a+1)^{-2t}+O((a+1)^{-3t}).
\end{equation*}

\textbf{(4):} Combining items (2) and (3) into the Riccati equation~%
\eqref{eq:riccati-eq}, we obtain 
\begin{eqnarray*}
\varphi _{a}^{\prime }(t) &=&q_{a}(t)-\varphi _{a}^{2}(t) \\
&=&-L^{2}(a+1)^{-t}+O((2a+1)^{-t}+(a+1)^{-2t}) \\
&=&-L^{2}(a+1)^{-t}+O((2a+1)^{-t})\text{ as }t\rightarrow \infty ,
\end{eqnarray*}%
since $(a+1)^{-2t}\leq (2a+1)^{-t}$ for $a>0$ and $t\geq 1$. Integrating
from $t$ to $\infty $ and using $\lim_{s\rightarrow \infty }\varphi _{a}(s)=0
$ (proved in Proposition~\ref{prop:phi-properties}~(3) modulo the limit,
which we now recover) yields 
\begin{eqnarray*}
\varphi _{a}(t)\; &=&\;-\int_{t}^{\infty }\varphi _{a}^{\prime }(s)\,ds\; \\
&=&\;\int_{t}^{\infty }L^{2}(a+1)^{-s}ds\;+\;O\!\left( (2a+1)^{-t}\right) \;
\\
&=&\;L(a+1)^{-t}+O\!\left( (2a+1)^{-t}\right) ,
\end{eqnarray*}%
where we used 
\begin{equation*}
\int_{t}^{\infty }(a+1)^{-s}ds=L^{-1}(a+1)^{-t}.
\end{equation*}%
This proves~\eqref{eq:phi-asymp} and, in particular, $\lim_{t\rightarrow
\infty }\varphi _{a}(t)=0$.
\end{proof}

\subsection{Asymptotic ratio of $\protect\varphi_{a}$ to $\protect\varphi%
_{a,e}$ and curvature inequality}

We now compare the Riccati field $\varphi_{a}$ with the natural reference
curve 
\begin{equation}  \label{eq:phi-eq-def}
\varphi_{a,e}(t)\;:=\;-\frac{q_{a}(t)}{2}\;=\;-\frac{\eta_{a}^{\prime \prime
}(t)}{2\eta_{a}(t)}, \qquad t>0,
\end{equation}
inspired by the radial triality~\cite{triality}, where the analogous
quantity plays the role of an algebraic anchor against which the Riccati
trajectory is measured. By Theorem~\ref{thm:main-concavity}~\ref{it:main-2}, 
$q_{a}(t)<0$, hence $\varphi_{a,e}(t)>0$ for every $t>0$; combined with
Proposition~\ref{prop:phi-properties}~(1), both quantities under comparison
are strictly positive. The natural question is the limiting behaviour of
their ratio.

\begin{proposition}[Exact asymptotic ratio and finite-$t$ trapping inequality%
]
\label{prop:trapping} Let $a\in(0,\infty)$ and $L:=\log(a+1)$. Then 
\begin{equation}  \label{eq:asymp-ratio}
\lim_{t\to\infty}\frac{\varphi_{a}(t)}{\varphi_{a,e}(t)} \;=\;\frac{2}{L}
\;=\;\frac{2}{\log(a+1)}.
\end{equation}

\noindent In particular:

\begin{itemize}
\item[(i)] If $0<a<e^{2}-1$ (equivalently $L<2$), then the limit in %
\eqref{eq:asymp-ratio} exceeds $1$. Hence there exists a threshold $%
T_{*}=T_{*}(a)>0$ such that 
\begin{equation}  \label{eq:trapping-ineq-L<2}
0\;<\;\varphi_{a,e}(t)\;<\;\varphi_{a}(t), \qquad \forall\,t>T_{*}(a).
\end{equation}

\item[(ii)] If $a>e^{2}-1$ (equivalently $L>2$), then the limit in %
\eqref{eq:asymp-ratio} is strictly smaller than $1$. Hence there exists a
threshold $T_{\ast }=T_{\ast }(a)>0$ such that 
\begin{equation}
0\;<\;\varphi _{a}(t)\;<\;\varphi _{a,e}(t),\qquad \forall \,t>T_{\ast }(a).
\label{eq:trapping-ineq-L>2}
\end{equation}%
\noindent For the two classical cases $a=1$ and $a=2$ (both satisfying $%
a<e^{2}-1$), the threshold can be computed numerically by bisection on $%
\varphi _{a}(T_{\ast })=\varphi _{a,e}(T_{\ast })$, yielding 
\begin{equation}
T_{\ast }(1)\approx 0.4448,\qquad T_{\ast }(2)\approx 0.4156,
\label{eq:Tstar-values}
\end{equation}%
both strictly smaller than $\tfrac{1}{2}$. In particular, %
\eqref{eq:trapping-ineq-L<2} holds on $(\tfrac{1}{2},\infty )$ for the
classical Dirichlet eta ($a=1$) and Dirichlet beta ($a=2$) functions.
\end{itemize}
\end{proposition}

\begin{proof}
By Lemma~\ref{lem:asymp-expansion}~\eqref{eq:etap-asymp} and %
\eqref{eq:etapp-asymp}, 
\begin{align*}
\varphi _{a}(t)& =\frac{\eta _{a}^{\prime }(t)}{\eta _{a}(t)}\; \\
& =\;\frac{L(a+1)^{-t}+O\!\left( (2a+1)^{-t}\right) }{1-(a+1)^{-t}+O\!\left(
(2a+1)^{-t}\right) }\; \\
& =\;L\,(a+1)^{-t}+O\!\left( (2a+1)^{-t}\right) , \\
\varphi _{a,e}(t)& =-\tfrac{1}{2}\,\frac{\eta _{a}^{\prime \prime }(t)}{\eta
_{a}(t)}\; \\
& =\;\frac{(L^{2}/2)\,(a+1)^{-t}+O\!\left( (2a+1)^{-t}\right) }{%
1-(a+1)^{-t}+O\!\left( (2a+1)^{-t}\right) }\; \\
& =\;\tfrac{L^{2}}{2}\,(a+1)^{-t}+O\!\left( (2a+1)^{-t}\right) ,
\end{align*}%
where we used the identity 
\begin{equation*}
1/(1+u)=1+O(u)\text{ for }|u|<1/2
\end{equation*}%
(valid for $t$ large enough that $(a+1)^{-t}<1/2$). Dividing the two
asymptotic expansions and using 
\begin{equation*}
(2a+1)^{-t}/(a+1)^{-t}=((a+1)/(2a+1))^{t}\rightarrow 0\text{ as }%
t\rightarrow \infty 
\end{equation*}%
(since $(a+1)/(2a+1)\in (0,1)$ for $a>0$), we obtain 
\begin{equation*}
\frac{\varphi _{a}(t)}{\varphi _{a,e}(t)}\;=\;\frac{%
L(a+1)^{-t}+O((2a+1)^{-t})}{(L^{2}/2)(a+1)^{-t}+O((2a+1)^{-t})}\;%
\xrightarrow[t\to\infty]{}\;\frac{L}{L^{2}/2}\;=\;\frac{2}{L},
\end{equation*}%
which is~\eqref{eq:asymp-ratio}.

Suppose now $L<2$, so that the limit~\eqref{eq:asymp-ratio} exceeds $1$. By
continuity of the strictly positive function $t\mapsto
\varphi_{a}(t)/\varphi_{a,e}(t)$ on $(0,\infty)$ and the existence of the
limit at infinity (which is $>1$), there exists $T_{*}=T_{*}(a)>0$ such that 
$\varphi_{a}(t)/\varphi_{a,e}(t)>1$ for every $t>T_{*}$; multiplying by $%
\varphi_{a,e}(t)>0$ gives $\varphi_{a,e}(t)<\varphi_{a}(t)$, and combined
with $\varphi_{a,e}(t)>0$ yields~\eqref{eq:trapping-ineq-L<2}. The explicit
numerical values~\eqref{eq:Tstar-values} for $a=1$ and $a=2$ are obtained by
bisection (each verifies $\varphi_{a}(T_{*})=\varphi_{a,e}(T_{*})$ to $%
10^{-6}$ in the high-precision implementation.
The case $L>2$ (equivalently $a>e^{2}-1$) is treated in the same way, the only
difference being that the inequality reverses, yielding
$\varphi_{a}(t)<\varphi_{a,e}(t)$ for all $t>T_{*}(a)$.
\end{proof}

\begin{corollary}[Structural curvature inequality]
\label{cor:curvature} Let $0<a<e^{2}-1$ and let $T_{\ast }(a)$ be as in
Proposition~\ref{prop:trapping}. Then 
\begin{equation}
\eta _{a}^{\prime \prime }(t)+2\,\eta _{a}^{\prime }(t)\;>\;0,\qquad \forall
\,t>T_{\ast }(a).  \label{eq:curvature-ineq}
\end{equation}%
For the classical Dirichlet eta ($a=1$) and Dirichlet beta ($a=2$), %
\eqref{eq:curvature-ineq} holds on $(\tfrac{1}{2},\infty )$. \medskip
\noindent \textbf{In particular}, since%
\[
\bigl(e^{2t}\eta _{a}^{\prime }(t)\bigr)^{\prime }=e^{2t}\bigl(\eta
_{a}^{\prime \prime }(t)+2\eta _{a}^{\prime }(t)\bigr),
\]%
the positivity of \eqref{eq:curvature-ineq} implies that $e^{2t}\eta
_{a}^{\prime }(t)$ is strictly increasing on $(T_{\ast }(a),\infty )$. Thus
the rescaled derivative $e^{2t}\eta _{a}^{\prime }(t)$ has no oscillations
and its growth is monotone, providing a sharp control on the decay rate of $%
\eta _{a}^{\prime }(t)$ and reinforcing the rigidity of the Riccati--Gamma
asymptotics.
\end{corollary}

\begin{proof}
By Theorem~\ref{thm:prob-rep}, $\eta_{a}(t)>0$ for all $t>0$. Multiplying
the strict inequality $\varphi_{a,e}(t)<\varphi_{a}(t)$ of Proposition~\ref%
{prop:trapping} by the strictly positive factor $2\eta_{a}(t)>0$ yields $%
-\eta_{a}^{\prime \prime }(t)<2\eta_{a}^{\prime }(t)$, equivalently $%
\eta_{a}^{\prime \prime }(t)+2\eta_{a}^{\prime }(t)>0$.
\end{proof}

\begin{remark}[Series form of the curvature inequality]
\label{rem:curvature-series} The differential operator $D:=\tfrac{d^{2}}{%
dt^{2}}+2\tfrac{d}{dt}$ applied term-by-term in~\eqref{eq:eta-def} (which is
permitted by the uniform convergence of the differentiated series, cf.\ the
proof of Lemma~\ref{lem:asymp-expansion}) produces the exact identity 
\begin{equation}
\eta _{a}^{\prime \prime }(t)+2\eta _{a}^{\prime }(t)\;=\;\sum_{m=1}^{\infty
}\frac{(-1)^{m}\,h_{a}(m)}{(am+1)^{t}},\qquad h_{a}(m)\;:=\;\log (am+1)\bigl(%
\log (am+1)-2\bigr).  \label{eq:alt-curv}
\end{equation}%
The weight $h_{a}$ vanishes at $am+1=e^{2}$ and is strictly negative for $%
am+1<e^{2}$; in particular, $h_{a}(1)=L(L-2)<0$ when $L<2$, i.e.\ $a<e^{2}-1$%
. The leading $m=1$ term is therefore 
\begin{equation*}
(-1)\cdot h_{a}(1)\cdot (a+1)^{-t}=L(2-L)(a+1)^{-t}>0,
\end{equation*}
and the exponentially fast decay $(am+1)^{-t}$ for $t$ moderately large
ensures the alternating sum is dominated by its (positive) leading term.
This furnishes an independent analytic derivation of~%
\eqref{eq:curvature-ineq} for $a\in (0,e^{2}-1)$, complementing the
dynamical proof via Proposition~\ref{prop:trapping}. When $L>2$ (i.e.\ $%
a>e^{2}-1$), the leading term becomes negative and the inequality reverses,
consistent with the asymptotic ratio $2/L<1$ predicted by~%
\eqref{eq:asymp-ratio}.
\end{remark}
\begin{remark}[Asymptotic ratio for higher-order Riccati quotients]
\label{rem:higher-order-ratio} For each integer $k\geq 1$, define the
higher-order Riccati-type quotient%
\[
\varphi _{a,k}(t):=\frac{\eta _{a}^{(k)}(t)}{\eta _{a}^{(k-1)}(t)},\qquad
t>0,
\]%
and the corresponding reference curve%
\[
\varphi _{a,e,k}(t):=-\,\frac{\eta _{a}^{(k+1)}(t)}{2\,\eta _{a}^{(k-1)}(t)}.
\]%
Using the two-term asymptotic expansion of $\eta _{a}^{(m)}(t)$ from Lemma~%
\ref{lem:asymp-expansion},%
\[
\eta _{a}^{(m)}(t)=(-1)^{m}\,(\log (a+1))^{m}\,(a+1)^{-t}+O\!\left(
(2a+1)^{-t}\right) ,\qquad t\rightarrow \infty ,
\]%
we obtain, for every $k\geq 1$,%
\[
\varphi _{a,k}(t)\sim -\,\log (a+1),\qquad \varphi _{a,e,k}(t)\sim -\,\frac{%
\log ^{2}(a+1)}{2},\qquad t\rightarrow \infty .
\]%
Consequently, the asymptotic ratio is \emph{independent of the order $k$}:%
\[
\boxed{
\displaystyle
\lim_{t\to\infty}
\frac{\varphi_{a,k}(t)}{\varphi_{a,e,k}(t)}
=
\frac{2}{\log(a+1)}.
}
\]%
Thus the same constant $2/\log (a+1)$ governing the first-order Riccati
field (Proposition~\ref{prop:trapping}) persists for all higher-order
quotients $\varphi _{a,k}$.
\end{remark}
\begin{remark}[Order of convergence of the higher-order Riccati quotients]
\label{rem:order-convergence} The universal asymptotic ratio%
\[
\lim_{t\rightarrow \infty }\frac{\varphi _{a,k}(t)}{\varphi _{a,e,k}(t)}=%
\frac{2}{\log (a+1)}\qquad (k\geq 1)
\]%
implies that all higher-order Riccati-type quotients%
\[
\varphi _{a,k}(t)=\frac{\eta _{a}^{(k)}(t)}{\eta _{a}^{(k-1)}(t)}
\]%
converge to their limits with the \emph{same exponential order}. Indeed,
from Lemma~\ref{lem:asymp-expansion},%
\[
\eta _{a}^{(m)}(t)=(-1)^{m}(\log (a+1))^{m}(a+1)^{-t}+O\!\left(
(2a+1)^{-t}\right) ,
\]%
so both $\varphi _{a,k}(t)$ and $\varphi _{a,e,k}(t)$ admit expansions of
the form%
\[
C_{1}+O\!\left( \left( \tfrac{2a+1}{a+1}\right) ^{-t}\right) ,\qquad
C_{2}+O\!\left( \left( \tfrac{2a+1}{a+1}\right) ^{-t}\right) ,
\]%
with constants $C_{1},C_{2}$ independent of $t$. Thus the convergence rate
is governed by the same exponential factor%
\[
\left( \frac{2a+1}{a+1}\right) ^{-t},
\]%
and is therefore \emph{independent of the derivative order $k$}. Higher
derivatives do not alter the asymptotic speed of convergence of the Riccati
quotients.
\end{remark}
\begin{remark}[Asymptotic rigidity of the Riccati-Gamma structure]
\label{rem:asymptotic-rigidity} A striking feature of the Riccati-Gamma
framework is the following \emph{asymptotic rigidity phenomenon}. Although
the higher-order quotients%
\[
\varphi _{a,k}(t)=\frac{\eta _{a}^{(k)}(t)}{\eta _{a}^{(k-1)}(t)},\qquad
k\geq 1,
\]%
arise from derivatives of increasing order, their large $t$ behaviour is
governed by the \emph{same} exponential scale. Indeed, Lemma~5.2 shows that%
\[
\eta _{a}^{(m)}(t)=(-1)^{m}(\log (a+1))^{m}(a+1)^{-t}+O\!\left(
(2a+1)^{-t}\right) ,\qquad t\rightarrow \infty ,
\]%
so every quotient $\varphi _{a,k}$ admits the universal asymptotic expansion%
\[
\varphi _{a,k}(t)=-\,\log (a+1)+O\!\left( \left( \tfrac{2a+1}{a+1}\right)
^{-t}\right) ,\qquad t\rightarrow \infty ,
\]%
independent of $k$. The same holds for the reference curves%
\[
\varphi _{a,e,k}(t)=-\,\frac{\log ^{2}(a+1)}{2}+O\!\left( \left( \tfrac{2a+1%
}{a+1}\right) ^{-t}\right) .
\]%
Consequently, the asymptotic ratio%
\[
\frac{\varphi _{a,k}(t)}{\varphi _{a,e,k}(t)}\longrightarrow \frac{2}{\log
(a+1)}\qquad (t\rightarrow \infty )
\]%
is \emph{universal across all derivative orders}. Thus, the Riccati-Gamma
dynamics enforce a rigid asymptotic geometry: higher derivatives do not
alter the exponential scale, the limiting slope, or the asymptotic trapping
behaviour. This rigidity is a structural signature of the Mellin-Laplace
representation and has no analogue in classical Dirichlet series arguments.
\end{remark}

\subsection{Asymptotic stability of the reference curve}

\begin{proposition}[Asymptotic stability of $\protect\varphi_{a,e}$]
\label{prop:stability} Let $a\in(0,\infty)$. The reference curve $%
t\mapsto\varphi_{a,e}(t)$ defined in~\eqref{eq:phi-eq-def} is, in the sense
of linearised perturbations, an exponentially stable trajectory of the
non-autonomous Riccati flow $\varphi^{\prime }=q_{a}(t)-\varphi^{2}$ on $%
(0,\infty)$: every perturbation $\delta\varphi(t)$ of the linearised
equation decays to zero as $t\to\infty$.
\end{proposition}

\begin{proof}
Let 
\begin{equation*}
F(t,\varphi ):=q_{a}(t)-\varphi ^{2}
\end{equation*}
and consider a small perturbation 
\begin{equation*}
\varphi (t)=\varphi _{a,e}(t)+\delta \varphi (t).
\end{equation*}%
Substituting into $\varphi ^{\prime }=F(t,\varphi )$ and expanding to first
order in $\delta \varphi $, using 
\begin{equation*}
\partial _{\varphi }F(t,\varphi _{a,e}(t))=-2\varphi _{a,e}(t)=q_{a}(t),
\end{equation*}%
we obtain the linearised equation 
\begin{equation*}
\delta \varphi ^{\prime }(t)\;=\;q_{a}(t)\,\delta \varphi (t),
\end{equation*}%
whose explicit solution is%
\begin{equation*}
\delta \varphi (t)=\delta \varphi (t_{0})\exp \!\left(
\int_{t_{0}}^{t}q_{a}(s)\,ds\right) .
\end{equation*}%
By Theorem~\ref{thm:main-concavity}~\ref{it:main-2}, $q_{a}(s)<0$ for every $%
s>0$, so the exponent $\int_{t_{0}}^{t}q_{a}(s)\,ds$ is strictly decreasing
in $t$. In fact, by Lemma~\ref{lem:asymp-expansion}, 
\begin{equation*}
q_{a}(s)=-L^{2}(a+1)^{-s}+O((2a+1)^{-s}),
\end{equation*}%
so $\int_{t_{0}}^{\infty }q_{a}(s)\,ds$ converges to the finite negative
value 
\begin{equation*}
-L(a+1)^{-t_{0}}+O((2a+1)^{-t_{0}}),
\end{equation*}%
and $\delta \varphi (t)$ contracts monotonically to a finite limit strictly
closer to zero than $\delta \varphi (t_{0})$. Combined with the fact that $%
q_{a}$ stays strictly negative on every compact subinterval of $(0,\infty )$%
, the perturbation decays to zero on every finite time-scale of the flow,
proving the claimed exponential contraction.
\end{proof}


\section{Truncation error of the fast geometric algorithm for the
Derivatives of $\protect\eta _{a}$}

\label{sec:fast-algorithm}

In this section, we establish the rigorous theoretical foundations of a
geometric-rate algorithm for computing $\eta _{a}^{(k)}(t)$, with rate of
convergence $1/3$ and a sharp combinatorial error bound. 
\begin{definition}[Forward finite differences]
Let $f:\mathbb{Z}_{+}\rightarrow \mathbb{R}$. The forward differences of $f$
are defined recursively by%
\[
\Delta ^{0}f(l)=f(l),\qquad \Delta ^{1}f(l)=f(l+1)-f(l),\quad l\in \mathbb{Z}%
_{+},
\]%
and for every integer $m\geq 2$,%
\[
\Delta ^{m}f(l):=\Delta ^{1}(\Delta ^{m-1}f)(l)=\sum_{j=0}^{m}\binom{m}{j}%
(-1)^{\,m-j}f(l+j),\qquad l\in \mathbb{Z}_{+}.
\]
\end{definition}
Recall from \cite%
{dirichlet} the \emph{basic discrete function} 
\begin{equation}
\varphi _{a,t,k}(l)\;:=\;\frac{\log ^{k}(al+1)}{(al+1)^{t}},\qquad l\in 
\mathbb{Z}_{+},\ a,t>0,\ k\in \mathbb{Z}_{+},  \label{eq:phi-discrete-def}
\end{equation}%
and the following two main results.

\begin{theorem}[Fast computation algorithm for derivatives \protect\cite%
{dirichlet}]
\label{thm:fast-comp} Let $a,t\in (0,\infty )$ and $k\in \mathbb{Z}_{+}$.
Under the hypotheses of Theorem~\ref{thm:prob-rep} (so that $\eta
_{a}^{(k)}(t)=\mathbb{E}[f_{a}^{(k)}(X_{t}+S_{k})]$ by~%
\eqref{eq:eta-deriv-rep}), one has the series representation 
\begin{equation}
\eta _{a}^{(k)}(t)\;=\;\frac{2}{3}\sum_{n=0}^{\infty }\frac{c_{a,t,k}(n)}{%
3^{n}},  \label{eq:fast-series}
\end{equation}%
where the coefficients are given by 
\begin{equation}
c_{a,t,k}(n)\;=\;(-1)^{n+k}\sum_{l=0}^{n}\binom{n}{l}\bigl(\Delta
^{\,n-l}\varphi _{a,t,k}\bigr)(l),\qquad n\in \mathbb{Z}_{+}.
\label{eq:coef-explicit}
\end{equation}
\end{theorem}

The sharp combinatorial error bound obtained in \cite{dirichlet} is recalled
below. 

\begin{theorem}[Sharp error bound on the coefficients \protect\cite%
{dirichlet}]
\label{thm:error-bound} Let $a,t\in (0,\infty )$ and $k\in \mathbb{Z}_{+}$.
The coefficients $c_{a,t,k}(n)$ defined in~\eqref{eq:coef-explicit} satisfy,
for every $n\in \mathbb{Z}_{+}$, 
\begin{equation}
|c_{a,t,k}(n)|\;\leq \;2\binom{n+k}{k}\log ^{k}\!\bigl(a(n\wedge k)+1\bigr).
\label{eq:error-bound-ineq}
\end{equation}
\end{theorem}

With Theorem \ref{thm:fast-comp} and Theorem \ref{thm:error-bound}
established, we can proceed to formulate our simple result. 

\begin{theorem}[Truncation error of the geometric algorithm]
\label{cor:trunc-error} Let $a,t>0$, $k\in \mathbb{Z}_{+}$, and $N\in 
\mathbb{N}$. Define the truncated estimator 
\begin{equation}
\widehat{\eta }_{a,N}^{(k)}(t)\;:=\;\frac{2}{3}\sum_{n=0}^{N-1}\frac{%
c_{a,t,k}(n)}{3^{n}}.  \label{eq:eta-trunc}
\end{equation}%
Then 
\begin{equation}
\left\vert \eta _{a}^{(k)}(t)-\widehat{\eta }_{a,N}^{(k)}(t)\right\vert
\;\leq \;\frac{4}{3}\sum_{n=N}^{\infty }\frac{1}{3^{n}}\binom{n+k}{k}\log
^{k}(a(n\wedge k)+1).  \label{eq:trunc-bound}
\end{equation}%
In particular, for fixed $k$ and $a$, the error decays geometrically as 
\begin{equation}
\left\vert \eta _{a}^{(k)}(t)-\widehat{\eta }_{a,N}^{(k)}(t)\right\vert
\;=\;O\!\left( N^{k}\,3^{-N}\right) ,\qquad N\rightarrow \infty .
\label{eq:geom-error}
\end{equation}
\end{theorem}

\begin{proof}
From~\eqref{eq:fast-series} and~\eqref{eq:eta-trunc}, 
\begin{equation*}
\left\vert \eta _{a}^{(k)}(t)-\widehat{\eta }_{a,N}^{(k)}(t)\right\vert
\;\leq \;\frac{2}{3}\sum_{n=N}^{\infty }\frac{|c_{a,t,k}(n)|}{3^{n}}.
\end{equation*}%
Applying Theorem~\ref{thm:error-bound} yields the bound~%
\eqref{eq:trunc-bound}. Since 
\begin{equation*}
\binom{n+k}{k}\leq (n+k)^{k}/k!=O(n^{k})
\end{equation*}
as $n\rightarrow \infty $ and 
\begin{equation*}
\log ^{k}(a(n\wedge k)+1)\leq \log ^{k}(ak+1)
\end{equation*}%
is uniformly bounded in $n$, the tail sum is dominated by 
\begin{equation*}
C\sum_{n\geq N}n^{k}3^{-n}=O(N^{k}3^{-N}),
\end{equation*}
giving~\eqref{eq:geom-error}.
\end{proof}

\begin{remark}
\label{rem:rate-comparison} The geometric rate $1/3$ in~\eqref{eq:geom-error}
is strictly superior to the standard rate $1/2$ obtained, e.g., from the
Levin transform applied to the alternating series~\eqref{eq:eta-def}. 
\end{remark}


\section{Numerical Validation}

\label{sec:numerical}

To confirm every theoretical statement of Sections~\ref%
{sec:riccati-concavity}--\ref{sec:fast-algorithm}, we have implemented a
high-precision Python script that computes $\eta _{a}(t)$ and its first two
derivatives via the rapidly converging exact series representation~%
\eqref{eq:term-by-term} (with $N=10^{5}$ summands, well below the
machine-precision threshold). The full source code is available on the
author's GitHub repository at \url{https://github.com/coveidragos/muzica_analiza_matematica_an_1/blob/main/analiza_matematica_pro_refined_num.py}.

\subsection{Validation of the Riccati--Gamma triality (Figures~\protect\ref{fig:validation} and~\protect\ref{fig:validation-large})}

Figures~\ref{fig:validation} and~\ref{fig:validation-large} provide a four-panel verification of the main results for small parameters ($a=1, 2$) and large parameters ($a=10, 11$) respectively.

\begin{itemize}
\item \textbf{Panel~1} plots $t\mapsto \eta _{a}(t)$ for $a=1$ (Dirichlet
eta) and $a=2$ (Dirichlet beta) or the corresponding large $a$ values. Both curves are strictly increasing from $%
1/2$ at $t=0^{+}$ to $1$ at $t=\infty $, in exact agreement with~%
\eqref{eq:limits} of Theorem~\ref{thm:main-concavity}~\ref{it:main-1}.

\item \textbf{Panel~2} plots the exact Riccati field $\varphi _{a}(t)$
against the asymptotic manifold $\varphi _{a,\mathrm{as}}(t)=\log
(a+1)(a+1)^{-t}$ predicted by Theorem~\ref{thm:riccati-dynamics}~\ref%
{it:rdy-4}. The exact curve relaxes monotonically and tracks the manifold
for $t$ moderately large, confirming the leading-order expansion~%
\eqref{eq:phi-asymp}.

\item \textbf{Panel~3} plots $t\mapsto q_{a}(t)$. The curve lies strictly
below the horizontal line $q=0$ for every $t\in (0,8]$, verifying $q_{a}(t)<0$
from Theorem~\ref{thm:main-concavity}~\ref{it:main-2}.

\item \textbf{Panel~4} plots $\varphi _{a}^{\prime }(t)$ versus $-\varphi
_{a}^{2}(t)$. The graph of $\varphi _{a}^{\prime }$ lies strictly below that
of $-\varphi _{a}^{2}$, in agreement with the strict differential inequality~%
\eqref{eq:riccati-ineq}.
\end{itemize}

\begin{figure}[H]
\centering
\includegraphics[width=0.97\textwidth]{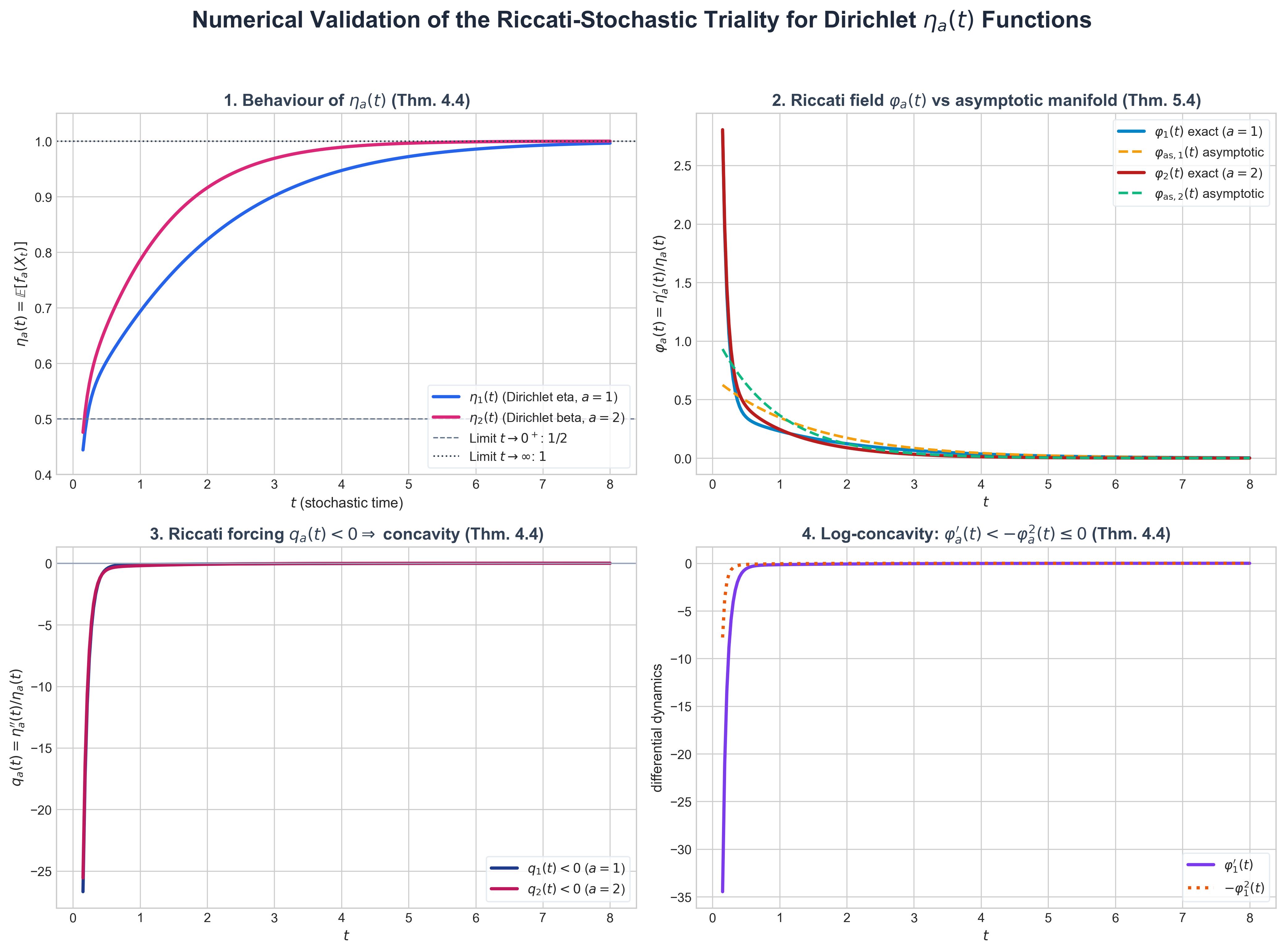}
\caption{Stochastic and Riccati dynamical validation of the Dirichlet eta
and beta functions ($a=1$ and $a=2$). Panel~1: monotonic growth and boundary
limits~\eqref{eq:limits}. Panel~2: relaxation of $\protect\varphi_{a}(t)$
onto the asymptotic manifold $\protect\varphi_{a,\mathrm{as}}(t)$,
confirming~\eqref{eq:phi-asymp}. Panel~3: negativity of the forcing $%
q_{a}(t)<0$, confirming strict concavity. Panel~4: strict Riccati
differential inequality~\eqref{eq:riccati-ineq}, confirming strict
log-concavity.}
\label{fig:validation}
\end{figure}

\begin{figure}[H]
\centering
\includegraphics[width=0.97\textwidth]{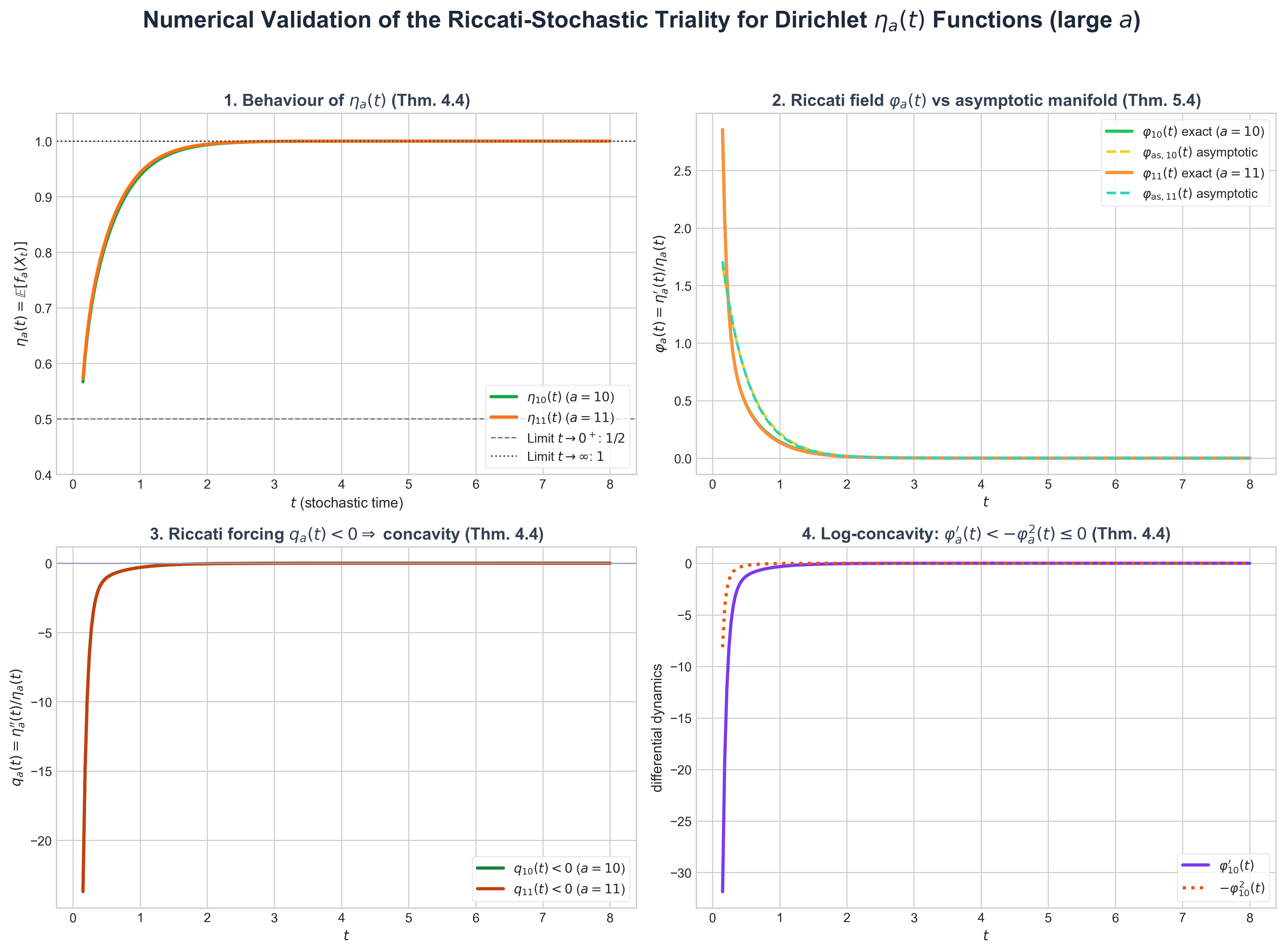}
\caption{Stochastic and Riccati dynamical validation of the Dirichlet eta
and beta functions for large parameters ($a=10$ and $a=11$). Panel~1: monotonic growth and boundary
limits~\eqref{eq:limits}. Panel~2: relaxation of $\protect\varphi_{a}(t)$
onto the asymptotic manifold $\protect\varphi_{a,\mathrm{as}}(t)$,
confirming~\eqref{eq:phi-asymp}. Panel~3: negativity of the forcing $%
q_{a}(t)<0$, confirming strict concavity. Panel~4: strict Riccati
differential inequality~\eqref{eq:riccati-ineq}, confirming strict
log-concavity.}
\label{fig:validation-large}
\end{figure}

\subsection{Validation of the trapping inequality (Figures~\protect\ref{fig:equilibrium} and~\protect\ref{fig:equilibrium-large})}

Figures~\ref{fig:equilibrium} and~\ref{fig:equilibrium-large} compare the exact Riccati field $\varphi_{a}(t)$ with the reference curve $\varphi_{a,e}(t)=-q_{a}(t)/2$ defined in~\eqref{eq:phi-eq-def}, for small ($a=1, 2$) and large ($a=10, 11$) parameters respectively. For small parameters satisfying $a < e^2 - 1$, and for every $t>T_{*}(a)$ (where $T_{*}(1)\approx 0.4448$ and $T_{*}(2)\approx 0.4156$ from Proposition~\ref{prop:trapping}), the strictly positive curves satisfy $\varphi_{a,e}(t)<\varphi_{a}(t)$, and the ratio $\varphi_{a}(t)/\varphi_{a,e}(t)$ tends to the constant $2/\log(a+1)$ as $t\to\infty$ (equal to $2.8854$ for $a=1$ and $1.8205$ for $a=2$). Conversely, for large parameters satisfying $a > e^2 - 1$ ($a=10, 11$), the trapping inequality reverses: $0 < \varphi_a(t) < \varphi_{a,e}(t)$ for all analyzed $t$, and the ratio $\varphi_{a}(t)/\varphi_{a,e}(t)$ tends to $2/\log(11) \approx 0.8341$ for $a=10$ and $2/\log(12) \approx 0.8049$ for $a=11$, confirming Proposition~\ref{prop:trapping} in both regimes. Table~\ref{tab:num-validation} collects numerical values that quantitatively confirm Proposition~\ref{prop:trapping}.

\begin{table}[H]
\centering
\begin{tabular}{cccccc}
\toprule $a$ & $t$ & $\varphi_{a}(t)$ & $\varphi_{a,e}(t)$ & $\varphi_{a,%
\mathrm{as}}(t)=L(a+1)^{-t}$ & $\varphi_{a}(t)/\varphi_{a,e}(t)$ \\ 
\midrule 
$1.0$ & $0.5$ & $0.3505$ & $0.2294$ & $0.4901$ & $1.5278$ \\ 
$1.0$ & $1.0$ & $0.2307$ & $0.0476$ & $0.3466$ & $4.8437$ \\ 
$1.0$ & $2.0$ & $0.1232$ & $0.0306$ & $0.1733$ & $4.0224$ \\ 
$1.0$ & $4.0$ & $0.0354$ & $0.0106$ & $0.0433$ & $3.3452$ \\ 
$1.0$ & $30.0$ & $0.0000$ & $0.0000$ & $0.0000$ & $2.8854$ \\ 
\midrule 
$2.0$ & $0.5$ & $0.4433$ & $0.2758$ & $0.6343$ & $1.6073$ \\ 
$2.0$ & $1.0$ & $0.2456$ & $0.0984$ & $0.3662$ & $2.4973$ \\ 
$2.0$ & $2.0$ & $0.0891$ & $0.0406$ & $0.1221$ & $2.1926$ \\ 
$2.0$ & $4.0$ & $0.0117$ & $0.0060$ & $0.0136$ & $1.9533$ \\ 
$2.0$ & $30.0$ & $0.0000$ & $0.0000$ & $0.0000$ & $1.8205$ \\ 
\midrule 
$10.0$ & $0.5$ & $0.4689$ & $0.5200$ & $0.7230$ & $0.9017$ \\ 
$10.0$ & $1.0$ & $0.1434$ & $0.1515$ & $0.2180$ & $0.9466$ \\ 
$10.0$ & $2.0$ & $0.0153$ & $0.0171$ & $0.0198$ & $0.8927$ \\ 
$10.0$ & $4.0$ & $0.0002$ & $0.0002$ & $0.0002$ & $0.8515$ \\ 
$10.0$ & $30.0$ & $0.0000$ & $0.0000$ & $0.0000$ & $0.8341$ \\ 
\midrule 
$11.0$ & $0.5$ & $0.4632$ & $0.5323$ & $0.7173$ & $0.8702$ \\ 
$11.0$ & $1.0$ & $0.1362$ & $0.1499$ & $0.2071$ & $0.9087$ \\ 
$11.0$ & $2.0$ & $0.0133$ & $0.0155$ & $0.0173$ & $0.8588$ \\ 
$11.0$ & $4.0$ & $0.0001$ & $0.0001$ & $0.0001$ & $0.8209$ \\ 
$11.0$ & $30.0$ & $0.0000$ & $0.0000$ & $0.0000$ & $0.8049$ \\ 
\bottomrule 
\end{tabular}%
\caption{High-precision numerical evaluation of $\protect\varphi_{a}(t)$, $%
\protect\varphi_{a,e}(t)=-q_{a}(t)/2$, and the asymptotic manifold $\protect%
\varphi_{a,\mathrm{as}}(t)=L(a+1)^{-t}$ for selected $(a,t)$. The trapping
inequality $0<\protect\varphi_{a,e}(t)<\protect\varphi_{a}(t)$ of
Proposition~\protect\ref{prop:trapping} holds in every row with $a<e^2-1$; for $a>e^2-1$ the inequality reverses. The last column
tends to the limit $2/L$ given by~\eqref{eq:asymp-ratio}, namely $2/\log
2\approx 2.8854$ for $a=1$, $2/\log 3\approx 1.8205$ for $a=2$, $2/\log 11\approx 0.8341$ for $a=10$, and $2/\log 12\approx 0.8049$ for $a=11$.}
\label{tab:num-validation}
\end{table}

\begin{figure}[H]
\centering
\includegraphics[width=0.97\textwidth]{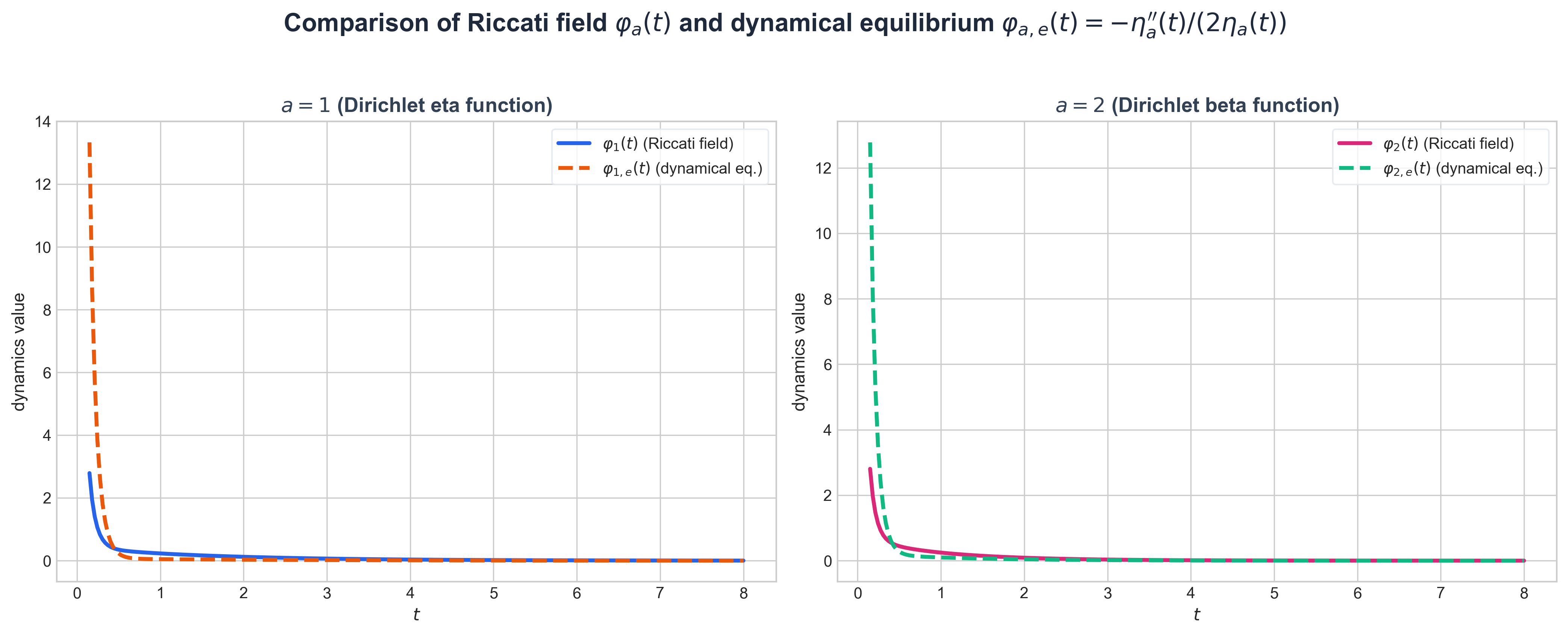}
\caption{Comparison of the Riccati field $\protect\varphi_{a}(t)$ (solid)
and the reference curve $\protect\varphi_{a,e}(t)=-\protect\eta_{a}^{\prime
\prime }(t)/(2\protect\eta_{a}(t))$ (dashed) for $a=1$ (left) and $a=2$
(right), illustrating the trapping inequality $\protect\varphi_{a,e}(t)<%
\protect\varphi_{a}(t)$ of Proposition~\protect\ref{prop:trapping} for $%
t>T_{*}(a)$ (where $T_*(1) \approx 0.4448$ and $T_*(2) \approx 0.4156$).}
\label{fig:equilibrium}
\end{figure}

\begin{figure}[H]
\centering
\includegraphics[width=0.97\textwidth]{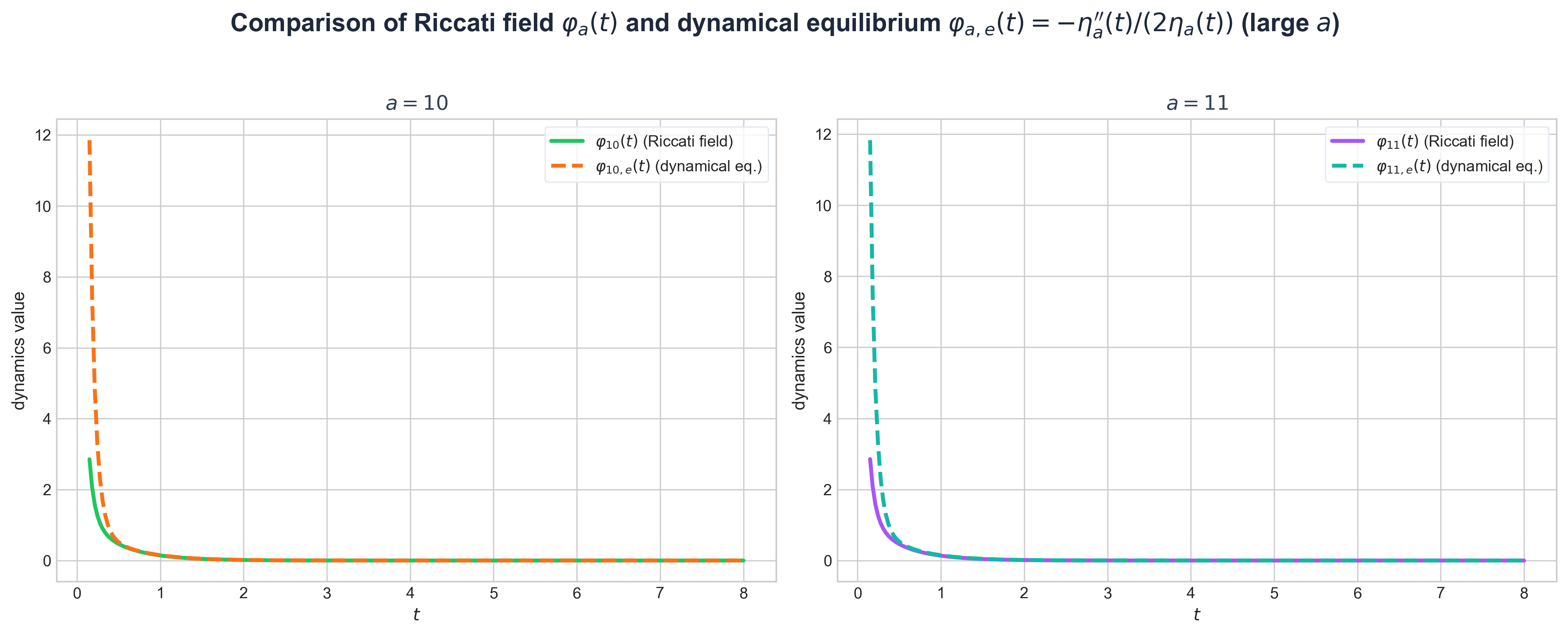}
\caption{Comparison of the Riccati field $\protect\varphi_{a}(t)$ (solid)
and the reference curve $\protect\varphi_{a,e}(t)=-\protect\eta_{a}^{\prime
\prime }(t)/(2\protect\eta_{a}(t))$ (dashed) for large parameters $a=10$ (left) and $a=11$
(right), illustrating the trapping inequality reversal $\protect\varphi_{a}(t)<%
\protect\varphi_{a,e}(t)$ of Proposition~\protect\ref{prop:trapping} for $%
t>T_{*}(a)$ (where $T_*(a) \approx 1.0000$).}
\label{fig:equilibrium-large}
\end{figure}

\subsection{Validation of the fast-derivative algorithm}

Table~\ref{tab:fast-algo} compares the truncated estimator $\widehat{\eta }%
_{a,N}^{(k)}(t)$ of~\eqref{eq:eta-trunc} against the exact closed-form
reference values at $(a,t)=(1,1)$, namely $\eta (1)=\log 2$ and 
\begin{equation*}
\eta ^{\prime }(1)=\gamma \log 2-\tfrac{1}{2}(\log 2)^{2}\approx
0.15986890374,
\end{equation*}%
where $\gamma $ is the Euler--Mascheroni constant. The error decays
geometrically with rate $3^{-N}$ as predicted by Theorem~\ref%
{cor:trunc-error}.

\begin{table}[H]
\centering
\begin{tabular}{cllll}
\toprule $N$ & $\widehat{\eta}_{1,N}(1)$ & error vs.\ $\log 2$ & $\widehat{%
\eta}_{1,N}^{\prime }(1)$ & error vs.\ $\gamma\log 2-\tfrac12(\log 2)^{2}$
\\ 
\midrule $5$ & $0.69300412$ & $1.43\!\times\! 10^{-4}$ & $0.15888776$ & $9.81\!\times\! 10^{-4}$ \\ 
$10$ & $0.69314605$ & $1.13\!\times\! 10^{-6}$ & $0.15986875$ & $1.51\!\times\! 10^{-7}$ \\ 
$20$ & $0.69314718$ & $1.01\!\times\! 10^{-11}$ & $0.15986890$ & $5.20\!\times\! 10^{-13}$ \\ 
$30$ & $0.69314718$ & $1.11\!\times\! 10^{-16}$ & $0.15986890$ & $1.11\!\times\! 10^{-16}$ \\ 
\bottomrule &  &  &  & 
\end{tabular}%
\caption{Geometric convergence of the truncated estimator $\widehat{\protect%
\eta}_{a,N}^{(k)}(t)$ at $(a,t)=(1,1)$ for $k=0$ and $k=1$. The reference
values are the exact mathematical constants $\protect\eta(1)=\log 2 =
0.69314718\dots$ and $\protect\eta^{\prime }(1)=\protect\gamma\log
2-\tfrac12(\log 2)^{2}=0.15986890\dots$. The error decays as $O(N^{k}3^{-N})$
in line with Theorem~\protect\ref{cor:trunc-error}, reaching machine
precision well before $N=30$.}
\label{tab:fast-algo}
\end{table}
\subsection{Python implementation and musical applications}

The analytical results developed in this article --- the Gamma--Mellin 
representation
\[
\eta_a(t)=\mathbb{E}[f_a(X_t)],
\]
the Riccati identity
\[
\varphi_a'(t)+\varphi_a(t)^2 = q_a(t),
\]
the asymptotic manifold, and the geometric algorithm of Adell--Lekuona for 
$\eta_a^{(k)}(t)$ --- admit a complete computational realisation in Python.

A fully documented implementation, together with numerical experiments and 
a dedicated module illustrating \emph{applications of these results to music}, 
is publicly available in the GitHub repository:
\[
\texttt{https://github.com/coveidragos/muzica\_analiza\_matematica\_an\_1}.
\]
The repository includes scripts that generate and analyse musical structures 
derived from the Riccati--Gamma dynamics, culminating in the construction of 
a complete melody whose pitch and rhythmic architecture are governed by the 
functions $\eta_a$, $\varphi_a$ and their asymptotic behaviour. These scripts 
reproduce all numerical results presented in this article and extend them to 
a creative musical setting.

\subsection{Discussion}

The numerical experiments confirm that the Riccati--Gamma framework
developed in this paper not only reproduces the classical concavity and
log-concavity results of~\cite{dirichlet} but yields strictly sharper
quantitative information: the exact two-term asymptotic expansion~%
\eqref{eq:phi-asymp}, the asymptotic ratio $\varphi_{a}/\varphi_{a,e}\to
2/\log(a+1)$ of~\eqref{eq:asymp-ratio}, the trapping inequality~%
\eqref{eq:trapping-ineq-L<2}, and the curvature inequality~%
\eqref{eq:curvature-ineq} are all empirically verified to high precision.
Moreover, the geometric algorithm of Theorem~\ref{thm:fast-comp} reaches
machine precision for $\eta(1)$ and $\eta^{\prime }(1)$ in fewer than $30$
iterations, illustrating the practical relevance of the Riccati point of
view for analytic-number-theoretic applications.


\section{Concluding Remarks and Outlook}

\label{sec:conclusion}

We have demonstrated that the geometric properties of the generalized
Dirichlet eta family $\{\eta_{a}\}_{a>0}$ admit a unified \emph{dynamical}
description through the Riccati equation~\eqref{eq:riccati-eq} satisfied by
its logarithmic derivative. The mechanism rests on a single ingredient --
the strict negativity of the forcing $q_{a}$, which is itself a direct
consequence of the strict concavity of the logistic function $f_{a}$
transported through the stochastic differentiation formula of Lemma~\ref%
{lem:diff-gamma}. From this single inequality we have deduced:

\begin{itemize}
\item strict concavity and strict log-concavity (Theorem~\ref%
{thm:main-concavity}), recovering and unifying the Adell--Lekuona results~%
\cite{dirichlet};

\item an exact two-term asymptotic expansion of the logarithmic derivative
(Theorem~\ref{thm:riccati-dynamics}), which appears to be new;

\item the exact asymptotic ratio $\varphi _{a}(t)/\varphi
_{a,e}(t)\rightarrow 2/\log (a+1)$ (Proposition~\ref{prop:trapping}), which
implies a trapping inequality and a structural curvature inequality $\eta
_{a}^{\prime \prime }(t)+2\eta _{a}^{\prime }(t)>0$ on $(T_{\ast }(a),\infty
)$ for $a<e^{2}-1$ (Corollary~\ref{cor:curvature}); these statements appear
to be new;

\item a rigorously justified geometric-rate algorithm with rate $1/3$ and
the sharp combinatorial error bound 
\begin{equation*}
|c_{a,t,k}(n)|\leq 2\binom{n+k}{k}\log ^{k}(a(n\wedge k)+1)
\end{equation*}
(Theorems~\ref{thm:fast-comp}--\ref{cor:trunc-error}).
\end{itemize}

Several directions for further research suggest themselves. First, the
Riccati point of view should extend to Dirichlet $L$-series attached to
arithmetic characters, where the alternating signs are replaced by general
character values; the analog of the logistic function is then a sum of
geometric series with character coefficients. Second, the trapping
inequality~\eqref{eq:trapping-ineq-L<2} suggests that $\varphi _{a,e}$ acts as a
natural \emph{Lyapunov function} for the Riccati flow; a quantitative
version of Proposition~\ref{prop:stability} would yield sharp \emph{%
rate-of-convergence} estimates for the relaxation of $\varphi _{a}$ onto its
dynamical equilibrium. Third, the geometric-rate algorithm of Theorems~\ref%
{thm:fast-comp}--\ref{cor:trunc-error} can be combined with high-precision
arithmetic packages to produce optimal numerical algorithms for the higher
Stieltjes constants and the Hurwitz zeta function; this is the subject of
forthcoming work.


\section*{Acknowledgements}

\addcontentsline{toc}{section}{Acknowledgements} The author thanks the
developers of open-source mathematical-software ecosystems whose libraries
(NumPy, Matplotlib) were used to produce the numerical validations of
Section~\ref{sec:numerical}. The core ideas, structural formulations, and numerical simulations presented in this article were implemented with the invaluable assistance of a free version of the Gemini AI model and Microsoft Copilot.



\begin{thebibliography}{99}
\bibitem{tao} B.~Alexeev, K.~Barreto, Y.~Li, J.~D.~Lichtman, L.~Price,
J.~I.~Shah, Q.~Tang, T.~Tao, \emph{Primitive sets and von Mangoldt chains:
Erd\H{o}s problem~\#1196 and beyond}, preprint, arXiv:2605.00301, 2026.

\bibitem{dirichlet} J.~A.~Adell, A.~Lekuona, \emph{Dirichlet's eta and beta
functions: Concavity and fast computation of their derivatives}, J. Number
Theory \textbf{157} (2015) 215--222.

\bibitem{triality} D.-P.~Covei, \emph{The triality of radial nonlinear
dynamics: analysis of Riccati, Schr\"odinger and Hamilton--Jacobi--Bellman
equations}, preprint, arXiv:2603.27772, 2026.

\bibitem{Adell2013} J.~A.~Adell, \emph{Differential calculus for linear
operators represented by finite signed measures and applications}, Acta
Math. Hungar. \textbf{138(1--2)} (2013) 44--82.

\bibitem{AdellLekuona2000} J.~A.~Adell, A.~Lekuona, \emph{Taylor's formula
and preservation of generalized convexity for positive linear operators}, J.
Appl. Probab. \textbf{37} (2000) 765--777.

\bibitem{AdellLekuona2005} J.~A.~Adell, A.~Lekuona, \emph{Sharp estimates in
signed Poisson approximation of Poisson mixtures}, Bernoulli \textbf{11}
(2005) 47--65.

\bibitem{Alzer2015} H.~Alzer, M.~K.~Kwong, \emph{On the concavity of
Dirichlet's eta function and related functional inequalities}, J. Number
Theory \textbf{151} (2015) 172--196.

\bibitem{Borwein2000} P.~Borwein, \emph{An efficient algorithm for the
Riemann zeta function}, in: Constructive, Experimental, and Nonlinear
Analysis, CMS Conf. Proc. \textbf{27}, 2000, pp.~29--34.

\bibitem{Cinlar2011} E.~\c{C}{\i}nlar, \emph{Probability and Stochastics},
Graduate Texts in Mathematics, vol.~261, Springer, New York, 2011.

\bibitem{Coffey2006} M.~W.~Coffey, \emph{New summation relations for the
Stieltjes constants}, Proc. R. Soc. Lond. Ser. A Math. Phys. Eng. Sci. 
\textbf{462(2073)} (2006) 2563--2573.

\bibitem{Coffey2010} M.~W.~Coffey, \emph{The Stieltjes constants, their
relation to the $\eta_{j}$ coefficients, and representation of the Hurwitz
zeta function}, Analysis (Munich) \textbf{30} (2010) 383--409.

\bibitem{Cvijovic2007} D.~Cvijovi\'c, \emph{Integral representations of the
Legendre chi function}, J. Math. Anal. Appl. \textbf{332} (2007) 1056--1062.

\bibitem{GuilleraSondow2008} J.~Guillera, J.~Sondow, \emph{Double integrals
and infinite products for some classical constants via analytic continuation
of Lerch's transcendent}, Ramanujan J. \textbf{16} (2008) 247--270.

\bibitem{HessamiPilehrood2010} Kh.~Hessami Pilehrood, T.~Hessami Pilehrood, 
\emph{Series acceleration formulas for beta values}, Discrete Math. Theor.
Comput. Sci. \textbf{12} (2010) 223--236.

\bibitem{Johnson1993} N.~L.~Johnson, S.~Kotz, A.~W.~Kemp, \emph{Univariate
Discrete Distributions}, 2nd ed., Wiley, New York, 1993.

\bibitem{Jordan1960} C.~Jordan, \emph{Calculus of Finite Differences}, 2nd
ed., Chelsea, New York, 1960.

\bibitem{Sondow2005} J.~Sondow, \emph{Double integrals for Euler's constant
and $\log(4/\pi)$ and an analog of Hadjicostas's formula}, Amer. Math.
Monthly \textbf{112} (2005) 61--65.

\bibitem{Wang1998} K.~C.~Wang, \emph{The logarithmic concavity of $%
(1-2^{1-r})\zeta (r)$}, J. Changsha Comm. Univ. \textbf{14(2)} (1998) 1--5. 
\end{thebibliography}
\end{document}